\documentclass[11pt,a4paper,reqno]{amsart}
\usepackage{amsfonts}
\usepackage{amsthm}
\usepackage{amsmath}
\usepackage{amscd}
\usepackage[latin2]{inputenc}
\usepackage{t1enc}
\usepackage[mathscr]{eucal}
\usepackage{indentfirst}
\usepackage{graphicx}
\usepackage{graphics}
\usepackage{pict2e}
\usepackage{dsfont}
\usepackage{epic}
\numberwithin{equation}{section}
\usepackage[margin=2.9cm]{geometry}
\usepackage{epstopdf} 
\usepackage{amssymb}
\usepackage{setspace}
\usepackage{enumerate}
\usepackage{bigstrut}
\usepackage{multirow}
\usepackage{mathtools,xparse}
\usepackage{mathrsfs}
\usepackage{hyperref}
\newtheorem{theorem}{Theorem}[section]

\newtheorem{lemma}[theorem]{Lemma}
\newtheorem{proposition}[theorem]{Proposition}
\theoremstyle{definition}
\newtheorem{definition}[theorem]{Definition}
\newtheorem{remark}[theorem]{Remark}
\allowdisplaybreaks
\usepackage{chngcntr}

\def\be{\begin{equation}}
\def\ee{\end{equation}}
\def\bepro{\begin{proposition}}
\def\enpro{\end{proposition}}
\def\belemma{\begin{lemma}}
\def\enlemma{\end{lemma}}
\def\it{\textit}

\newcommand{\E}{\mathop{\mathbb{E}}}
\newcommand{\1}{\mathds{1}}

\newcommand{\R}{\mathbb{R}}
\newcommand{\norm}[1]{\left\lVert#1\right\rVert}

\begin{document}

\title{Stochastic differential equations with critical drifts}

\author{Kyeongsik Nam}

\begin{abstract}
We establish the well-posedness  of  SDE with the additive noise when a singular drift belongs to the critical spaces. We prove that if the drift belongs  to the Orlicz-critical space $L^{q,1}([0,T],L^p_x)$ for $p,q\in (1,\infty)$ satisfying $\frac{2}{q}+\frac{d}{p} =1$, then the corresponding SDE admits a unique strong solution. We also derive the Sobolev regularity of a solution under the Orlicz-critical condition.
\end{abstract}

\address{ Department of Mathematics, Evans Hall, University of California, Berkeley, CA
94720, USA} 

\email{ksnam@math.berkeley.edu}


 
\maketitle

\section{Introduction}

According to the classical theory in the ordinary differential equations (ODE), if the vector field $b(t,x)$ is uniformly Lipschitz continuous in $x$ and continuous in $t$, then there exists a unique solution $x(t)$ associated with ODE $x'(t)=b(t,x(t))$,  $x(t_0)=x_0$. In the absence of Lipschitz continuity in $x$, the existence or uniqueness may not hold. For instance, when $b(t,x)$ is just continuous in $x$, we only have the existence of a solution according to the classical Peano existence theorem. The example $b(t,x)=\sqrt{|x|}$ demonstrates the non-uniqueness of solutions to ODE. 

A breakthrough progress in this context was made by Diperna and Lions \cite{lions}. They introduced the theory of a \it{Lagrangian flow}, which generalizes the notion of a  classical flow associated with ODE. They proved that under a suitable integrability condition on $b$ and $\text{div}b$, which is weaker than Lipschitz continuity, it is possible to construct a Lagrangian flow associated to such ODE. This result was extended to the bounded variation (BV) vector fields by Ambrosio \cite{amb}. A key observation is the link between the Lagrangian flow of ODE and the continuity equation $\partial_t\mu +\text{div}(b\mu)=0$. Once the well-posedness of the continuity equation in $L^\infty([0,T],L^1(\R^d) \cap L^\infty(\R^d))$ can be proved for singular $b$, then one can construct a unique regular Lagrangian flow of ODE (see \cite{amb} for details). 

Once the noise is added to ODE, we have well-posedness results for a considerably larger class of drifts $b$. Consider the stochastic differential equation (SDE) of the following form:
\be \label{SDE}
\begin{cases}
dX_t=b(t,X_t)dt+dB_t, \quad 0\leq t\leq T, \\
X_0=x.
\end{cases}
\ee 
Here, $B_t$ denotes the standard Brownian motion on a filtered space $(\Omega, \mathcal{F}, \mathcal{F}_t, P)$.
 According to the classical theory by It\^o, SDEs with Lipschitz continuous drift and diffusion coefficients possess a unique strong solution. There have been numerous works to   extend this classical result to a broad class of singular coefficients. Veretennikov \cite{ver} obtained a satisfactory result when the additive SDE \eqref{SDE} has a bounded drift $b$ in the case of dimension one.  Krylov and R\"ockner \cite{krylov} made a breakthrough by establishing the well-posedness of SDE \eqref{SDE} under the condition:
\be \label{KR}
b\in L^q([0,T],L^p_x),\quad \text{for} \quad \frac{2}{q}+\frac{d}{p} <1, \ 1<p,q<\infty
\ee
($d$ denotes the dimension of the underlying space). This is a striking result considering that no regularity condition is imposed on the singular drift $b$, and $b$ does not needed to be bounded. The key ingredient to prove the well-posedness of \eqref{SDE} is a \it{Yamada-Watanabe principle} \cite{YW,YW1}: existence of a weak solution together with the uniqueness of a strong solution to \eqref{SDE} imply the existence of a strong solution and uniqueness of a weak solution to \eqref{SDE}. After this groundbreaking work, lots of the well-posedness results have been established for the various types of non-degenerate diffusion coefficients under the  condition  of type \eqref{KR}. For instance, Zhang \cite{zhang1} proved that SDE:
\begin{align*}
dX_t = b(t,X_t) dt+\sigma(t,X_t)dB_t,\quad X_0=x,
\end{align*}
admits a unique local strong solution when $\sigma$ is non-degenerate and $b$ belongs to the local $L^q([0,T],L^p_x)$ space:
\begin{align*}
b\in L^q([0,T],L^p_{\text{loc}}),\quad \text{for} \quad \frac{2}{q}+\frac{d}{p} <1, \ 1<p,q<\infty.
\end{align*}
 We refer to \cite{att, da, dav, malliavin, xz, zhang3}  for the further results in this direction.

As mentioned above, the well-posedness theory of SDEs \eqref{SDE} at the subcritical regime \eqref{KR} has been quite well-established. However, at the supercritical regime: \
\begin{align} \label{super}
b\in L^q([0,T],L^p_{\text{loc}}),\quad \text{for} \quad \frac{2}{q}+\frac{d}{p} >1, \ 1<p,q<\infty,
\end{align} 
SDE \eqref{SDE} may not be well-posed in general.
In fact, it is proved in  \cite[Section 7.2]{gub2} that for a singular drift $b$ given by
\begin{align} \label{singular} 
b(t,x) = -\beta\frac{x}{|x|^2} \1_{x\neq 0},\quad \beta>\frac{1}{2},
\end{align}
the corresponding SDE \eqref{SDE} with the initial condition $X_0=0$ does not admit a solution. Since a singular drift $b$ in  \eqref{singular} satisfies
\begin{align*}
b\in L^\infty([0,T],L^p_{\text{loc}})
\end{align*}
for any $p<d$, this counterexample shows that SDE \eqref{SDE} may not be well-posed at the supercritical regime \eqref{super}.  In other words,  the lack of integrability of a  singular drift may lead to the non-existence of a solution.
Therefore, this counterexample at the supercritical regime \eqref{super} and the previously known well-posed results at the subcritical regime \eqref{KR} demonstrate that the qualitative properties of SDE \eqref{SDE} depend delicately on the integrability condition on the singular drift $b$. 

However, to the best of author's knowledge, there have been no clear answers at the critical regime:
 \begin{align}\label{critical}
b\in L^q([0,T],L^p_x),\quad \text{for} \quad \frac{2}{q}+\frac{d}{p} =1, \ 1<p,q<\infty.
\end{align} 
In fact, it has been a long-standing conjecture whether or not SDE \eqref{SDE} is well-posed under the critical condition \eqref{critical}.

 The condition $\frac{2}{q}+\frac{d}{p}\leq 1$, including both the subcritical case \eqref{KR} and the critical case \eqref{critical}, is often referred to as Ladyzhenskaya-Prodi-Serrin (LPS) condition. The space $L^q([0,T],L^p_x)$ with $\frac{2}{q}+\frac{d}{p}\leq 1$ is a function space where the regularity of a solution to the 3D Navier-Stokes equations holds (see \cite{LPS1, fluid, LPS2, LPS3}). There have been several ways to study Navier-Stokes equations in the probabilistic point of view. For example, the stochastic Lagrangian representation of the 3D incompressible Navier-Stokes equations was studied by Constantin and Iyer \cite{nse} (see also \cite{constantin} for the Eulerian-Lagrangian description of Euler equations). They proved that for a sufficiently smooth divergence-free vector field $u_0$, if the pair $(u,X)$ satisfy the following stochastic system:
\begin{gather*}
dX=udt+\sqrt{2}dB, \\
u=\E \mathbf{P}\big[\nabla^T (X^{-1}) (u_0 \circ X^{-1})\big],
\end{gather*}
($\mathbf{P}$ is the Leray-Hodge projection on divergence-free vector fields), then $u$ satisfies the incompressible Navier-Stokes equations with an initial data $u_0$. Also, another probabilistic interpretation of a certain class of solutions to the Navier-Stokes equations using the Hamiltonian dynamics approach was found by Rezakhanlou \cite{re, re2}. These fundamental relationships between the Navier-Stokes equations and SDEs demonstrate that it is important to establish qualitative properties of the SDEs \eqref{SDE} with rough drifts, in particular when drifts $b$ satisfy the critical LPS condition \eqref{critical}.

As mentioned before, the well-posedness question of SDE \eqref{SDE}  at the critical regime \eqref{critical} has been  a longstanding conjecture.  One may wonder if the previously known arguments to prove the well-posedness of SDE \eqref{SDE} under the subcritical condition \eqref{KR} can be extended to the critical case \eqref{critical}. To the best of author's knowledge, all of the known arguments break down at the critical regime \eqref{critical}. For instance, Rezakhanlou \cite{re2} proved the existence of a strong solution to SDE \eqref{SDE} under the subcritical condition \eqref{KR} by controlling  the following quantity:
\begin{align} \label{fra}
\E\Big \vert \Big[ \int_{t_0<t_1<\cdots<t_n<t} \prod_{i=1}^n b^i_{\alpha_i} (t_i,x+B_{t_i})dt_1\cdots dt_n\Big] \Big \vert
\end{align} 
in terms of $\norm{b^i}_{L^q([0,T],L^p_x)}$'s ($b^1,\cdots,b^n$ are smooth functions, $\alpha_1,\cdots,\alpha_n$ are multi-indices with $|\alpha_i|=1$, and $b^i_{\alpha_i}$ denotes the partial derivative). In fact, by approximating the singular drift $b$ by smooth drifts, the upper bound of \eqref{fra} provides an enough compactness to obtain a solution to SDE \eqref{SDE}. However, the existing argument to control the quantity \eqref{fra} by $\norm{b^i}_{L^q([0,T],L^p_x)}$'s does not work under the critical case \eqref{critical}.

The methods  in  \cite{ff3,krylov,xz,zhang1} to prove the well-posedness of SDE \eqref{SDE} under the subcritical condition \eqref{KR} also break down  at  the critical regime \eqref{critical}. For instance,  the arguments used in \cite{ff3,krylov} to obtain the  Khasminskii-type estimate (see \cite{khas}), which is a key ingredient  to prove the existence of a weak solution to SDE \eqref{SDE}, highly rely on  the subcritical assumption \eqref{KR} (see Section \ref{section 3.1} for explanations).  The difficulties also arise when we try to obtain a priori estimate of solutions to the Kolmogorov PDE, which plays a crucial role in  proving the strong uniqueness of SDE \eqref{SDE}. At the critical regime,  this Kolmogorov PDE possesses a singular coefficient which belongs to the critical Lebesgue space \eqref{critical}. The lack of nice embedding properties for the mixed-norm Sobolev spaces at the critical regime causes difficulties to study this singular PDE under the critical condition \eqref{critical} (see Section \ref{section 3.2} for explanations). Even if these problems are resolved, several difficulties also emerge when proving the strong uniqueness of SDE \eqref{SDE} using the Zvonkin's transformation method \cite{zvo}.

Recently, an interesting result at the critical regime \eqref{critical} was obtained by Beck et al. \cite{gub2}. It is proved that for almost all realization $w$, one can construct a stochastic Lagrangian flow associate with SDE \eqref{SDE}. Here, $\phi:[0,T]\times \R^d \times \Omega \rightarrow \R^d$ is called a \it{stochastic Lagrangian flow} to \eqref{SDE} provided that the following conditions are satisfied: \\
(i) $w$-almost surely, $\phi(\cdot,\cdot,w)-B_t(w)$ is a Lagrangian flow to the random ODE: $x'(t)=b^w(t,x(t))$, where $b^w(t,x)=b(t,x+B_t(w))$. \\
(ii) If we denote $\mathcal{F}_t$ by a natural filteration of the Brownian motion $B_t$, then  $\phi$ is weakly progressively measurable with respect to $\mathcal{F}_t$.

The main ingredient of the proof is to study the following random divergence form PDE:
\be \label{random}
u_t^w+\text{div}(b^wu^w) = 0.
\ee 
It was proved in \cite{gub2} that $w$-almost surely, there exists a unique weak solution to \eqref{random} in a suitable function space. From this, authors proved the existence of a solution to SDE \eqref{SDE} for almost everywhere $x\in \R^d$, $w$-almost surely.

In this paper, we establish the  well-posedness  of SDE \eqref{SDE} for arbitrary starting point $x\in \R^d$ at the critical regime. More precisely, we prove that there exists a unique strong solution to SDE \eqref{SDE} for every $x\in \R^d$ when the Lebesgue-type $L^q$ integrability in a time variable is replaced with a slightly  stronger Lorentz-type $L^{q,1}$ integrability condition:
 \begin{align} \label{lor}
 b\in L^{q,1}([0,T],L^p_x) \quad \text{for} \quad \frac{2}{q}+\frac{d}{p} = 1, \ 1<p,q<\infty
\end{align} 
(see Theorem \ref{theorem1}). We refer to the condition \eqref{lor} as \it{Orlicz-critical} condition. Under this condition, we can resolve some of the difficulties that we encounter in the Lebesgue-critical case \eqref{critical} explained above (see Section \ref{section 2} for details). 
 
 To the best of author's knowledge, this is the first well-posedness result of SDE \eqref{SDE} for arbitrary starting point $x\in \R^d$ at the critical regime.  This well-posedness result of SDE \eqref{SDE} at the Orlicz-critical regime can be regarded as orthogonal   to the result in  \cite{gub2} at the Lebesgue-critical regime. In fact, in  \cite{gub2}, the existence of a solution is  proved for \it{almost everywhere} starting point $x\in \R^d$ under the Lebesgue-critical condition \eqref{critical}. On the other hand, the main result of this paper Theorem \ref{theorem1} claims that when a slightly more integrability condition is imposed on the time variable, SDE \eqref{SDE} admits a unique solution for \it{every} starting point $x\in \R^d$. As mentioned before, since SDE \eqref{SDE}  may not be well-posed at the supercritical regime (see the counterexample \eqref{singular}),  Theorem \ref{theorem1} below   provides an almost optimal well-posedness result.

Once the well-posedness of SDE \eqref{SDE} is established at the Orlicz-critical condition \eqref{lor}, the next natural and crucial task is studying qualitative properties of a solution to SDE \eqref{SDE}. Unlike the ODE, an interesting regularization effect happens when the noise is added to the ODE. In fact, the regularity of a flow associated to ODE $x'(t)=b(t,x(t))$ is not better than the regularity of $b$ in general. On the other hand, in the case of SDE \eqref{SDE}, it is proved by Flandoli et al. \cite{gub} that if $b\in L^\infty_t(C^\alpha_x)$ for $0<\alpha<1$, then a solution to \eqref{SDE} is almost surely $C^{1+\beta}$ for arbitrary $\beta<\alpha$. This  regularization effect also happens even when a singular drift $b$ has no regularity. For example, Fedrizzi and Flandoli \cite{ff3} obtained the Sobolev regularity of a solution of SDE \eqref{SDE} under the subcritical condition \eqref{KR}. They proved that the stochastic flow $\phi(0,t,x)$ associated with SDE \eqref{SDE} is differentiable in the following sense: for any elementary direction vector $e_i$,
\begin{align*}
\lim_{h \rightarrow 0}\frac{\phi(0,\cdot,x+he_i)-\phi(0,\cdot,x)}{h}
\end{align*} exists as a strong limit in $L^2(\Omega \times [0,T], \R^d)$. 

In the second part of this paper, we  establish the improved  regularity property of a solution to SDE \eqref{SDE} under the Orlicz-critical condition \eqref{lor}. We prove that a solution to SDE \eqref{SDE} possesses the Sobolev regularity, and its (spatial) weak derivative has a  nice integrability property (see Theorem \ref{theorem2}).

The paper is organized as follows. We state the main results of this paper in Section \ref{section 2}.  In Section \ref{section 3}, we prove that   SDE \eqref{SDE} is  well-posed at the Orlicz-critical regime \eqref{lor}. In Section \ref{section 4}, we derive the Sobolev regularity of a solution to SDE \eqref{SDE} under the Orlicz-critical condition \eqref{lor}. Finally, we introduce the key properties of the Lorentz spaces and some useful lemmas used in  Appendix \ref{section a}.  

Throughout this paper, $B_t$ and $B^x_t$ denote the Brownian motions starting from the origin and $x$, respectively. $\nabla$, $\Delta$, and $\mathcal{M}$ denote the gradient, Laplacian, and the Hardy-Littlewood maximal function. For two Banach spaces $X$ and $Y$, $[X,Y]_{\theta,q}$ denotes a real interpolation of $X$ and $Y$ with parameters $0<\theta<1$ and $q\in [1,\infty]$. Also, $f \lesssim_{\alpha} g$ means that $f \leq Cg$ for some constant $C=C(\alpha)$. We say  $f \sim_\alpha g$ provided that $f \lesssim_\alpha g$, $g  \lesssim_\alpha f$. Finally, for $d\times d$ matrix $A$, $|A|$ denotes a Hilbert-Schmidt norm.

\section{Main results} \label{section 2}
The first main result of this paper is the well-posedness result of SDE \eqref{SDE} from every starting point $x\in \R^d$ at the Orlicz-critical regime:
\begin{theorem}  \label{theorem1}
Suppose that the drift $b$ satisfies:
\be \label{condition}
b\in L^{q,1}([0,T],L^p_x) \quad \text{for} \quad \frac{2}{q}+\frac{d}{p} = 1, \ 1<p,q<\infty.
\ee 
Then, there exists a unique strong solution to  SDE  \eqref{SDE} for any $x\in \R^d$.
\end{theorem}

The proof of Theorem  \ref{theorem1} follows the arguments in \cite{ff3,krylov}, and uses the Yamada-Watanabe principle. We  prove the existence of a weak solution and the uniqueness of a strong solution separately. In both cases, one has to play with the Orlicz-critical condition \eqref{condition} in a  delicate way due to the critical nature of the exponents $p$ and $q$. In order to prove the weak existence, we need to obtain the exponential integrability of a certain stochastic process under the Orlicz-critical condition \eqref{condition}. This can be successfully done with the aid of Khasminskii's Lemma and the functional inequality for the Lorentz spaces (see Section \ref{section 3.1} for details).

Several difficulties  arise  when we  prove the strong uniqueness. The main problem comes from  the Kolmogorov equation possessing critical  coefficients. We first establish the new embedding  properties for the mized-norm Sobolev spaces at the Orlicz-critical regime \eqref{condition}, as an application of the O'Neil's convolution inequality for the mixed-norm Lorentz spaces (see Proposition \ref{2.7} and \ref{A.3}). Then, by obtaining an a priori estimate for the standard heat equation  using the interpolation theory,  we obtain a nice a priori estimate of a solution to the Kolmogorov equation. This  well-posedness result for the parabolic equations possessing critical singular coefficients is also one of the main accomplishments of the paper (see Section \ref{section 3.2} for details). Finally,  by deriving nice exponential integrability properties of a solution to SDE \eqref{SDE} at the Orlicz-critical regime (see Remark \ref{re2.16}), we can finally  prove the strong uniqueness of SDE \eqref{SDE}. This can be done by introducing a new auxiliary SDE transformed from the original SDE \eqref{SDE}, motivated by the Zvonkin's transformation method \cite{zvo} (see Section \ref{section 3.3} for details).

The second result of this paper is the Sobolev regularity of a solution to SDE \eqref{SDE} under the Orlicz-critical condition \eqref{condition}:

\begin{theorem} \label{theorem2}
There exists a stochastic flow $\phi(s,t,x)$ associated with SDE \eqref{SDE} under the condition \eqref{condition}. Also, for each $0\leq t\leq T$, $\phi(0,t,\cdot)$ is almost surely weakly differentiable and its weak derivative belongs to $L^\infty(\R^d, L^r(\Omega))$ for any $r\in [1,\infty)$. 
\end{theorem}

We prove the improved Sobolev regularity of a solution using the ideas in \cite{ff2}. More precisely, we obtain the regularity for the auxiliary SDE first, and then  derive the regularity properties of  the  original SDE \eqref{SDE}. The key steps are similar to \cite{ff2}, but we need to work in a delicate way due to the critical nature of the exponents $p$ and $q$ (see Section \ref{section 4.2} for details).

\section{Well-posedness result  at the Orlicz-critical regime } \label{section 3}
In this section, we construct a unique strong solution to SDE \eqref{SDE} under the Orlicz-critical condition \eqref{condition}. Thanks to the Yamada-Watanabe principle \cite{YW,YW1}, it reduces to establish the existence of a  weak solution and the uniqueness of a  strong solution to SDE \eqref{SDE}. We prove both of them separately under the Orlicz-critical condition \eqref{condition}. In Section \ref{section 3.1}, we show the existence of a  weak solution. In Section \ref{section 3.2}, we study the Kolmogorov PDE associated with SDE \eqref{SDE}, which is an essential ingredient to apply the Zvonkin's transformation method \cite{zvo} to obtain an auxiliary SDE. Section \ref{section 3.3} is devoted to prove the uniqueness of a strong solution to SDE \eqref{SDE}.
\subsection{Existence of a weak solution to SDE}\label{section 3.1} In this section, we construct a weak solution to SDE \eqref{SDE} under the Orlicz-critical condition \eqref{condition}.
Throughout this section, we assume that $B^x_t$ is a Brownian motion starting from $x$ with a natural filtration $\mathcal{F}_t$.  First, we recall the following key lemma by Khasminskii (see \cite{khas}):
\begin{lemma} \label{2.1} Suppose that a  nonnegative function $f$ satisfies
\begin{align*}
\sup_{x\in\R^d}\E\int_0^Tf(s,B^x_s)ds=M<1.
\end{align*}
Then, we have
\begin{align*}
\sup_{x\in\R^d}\E e^{\int^T_0f(s,B_s^x)ds} \leq \frac{1}{1-M}.
\end{align*}
\end{lemma}
The quantity $\sup_{x\in\R^d}\E\int_0^Tf(s,B^x_s)ds
$  in Lemma \ref{2.1} can be controlled for a large class of functions:
\begin{proposition} \label{2.2}
Suppose that two exponents $p,q\in (1,\infty)$ satisfying $\frac{2}{q}+\frac{d}{p} = 2$ are given. Then, for any $f\in L^{q,1}([0,T],L^p_x)$, 
\begin{align*}
\sup_{x\in\R^d}\E\int_0^Tf(s,B^x_s)ds < C\norm{f}_{L^{q,1}([0,T],L^{p}_x)}
\end{align*}
holds for some constant $C=C(p,q)$ independent of $f$ and $T$.
\end{proposition}
\begin{proof}
Let $p',q'$ be the conjugate exponents of $p,q$, respectively. Then, 
\begin{align*}
\E\int_0^T f(s,B^x_s)ds&=\int_0^T \int_{\R^d}(2\pi s)^{-\frac{d}{2}}f(s,x+y)e^{-\frac{|y|^2}{2s}}dyds \\
&\leq \int_0^T (2\pi s)^{-\frac{d}{2}} \norm{f(s,\cdot)}_{L^p_x}\norm{e^{-\frac{|\cdot|^2}{2s}}}_{L^{p'}_x}ds \\
&= K\int_0^T \norm{f(s,\cdot)}_{L^p_x} s^{d/2p'-d/2} ds \\
&\leq C \norm{f}_{L^{q,1}([0,T],L^p_x)}\norm{s^{-\frac{d}{2}(1-\frac{1}{p'})}}_{L^{q',\infty}([0,T])} \\
&=C\norm{f}_{L^{q,1}([0,T],L^p_x)}.
\end{align*}

Here, we used the fact that for some universal constant $K$, $\norm{e^{-\frac{|\cdot|^2}{2s}}}_{L^{p'}_x} = K\cdot s^\frac{d}{2p'}$ for all $s>0$ in the third line, and applied the H\"older's inequality for the Lorentz spaces in the fourth line (see Appendix \ref{section a}). Also, we used the fact  $\frac{d}{2}(1-\frac{1}{p'})=\frac{1}{q'}$ in order to conclude that $\norm{s^{-\frac{d}{2}(1-\frac{1}{p'})}}_{L^{q',\infty}([0,T])}=1$. 
\end{proof}
\begin{remark}
The analogous result  is proved in \cite{ff3, krylov} for $f\in L^{q}([0,T],L^p_x)$ with $\frac{2}{q}+\frac{d}{p}<2$. However,  at the critical regime $\frac{2}{q}+\frac{d}{p}=2$, the quantity $\norm{s^{-\frac{d}{2}(1-\frac{1}{p'})}}_{L^{q'}([0,T])}$ in the proof of Proposition \ref{2.2}  is not finite  due to the singularity at $s=0$. This quantity can be made finite by imposing a slightly stronger Lorentz integrability on the time variable of a function $f$.
\end{remark}
Proposition \ref{2.2}, combined with the Markov property and Lemma \ref{2.1}, implies the following proposition.
\begin{proposition} \label{2.3}
Suppose that $f\in L^{q,1}([0,T],L^p_x)$ for $p,q\in (1,\infty)$ satisfying $\frac{2}{q}+\frac{d}{p} = 2$. Then,  the following quantity is finite: 
\be \label{eq2.1}
\sup_{x\in\R^d}\E e^{\int^T_0f(s,B_s^x)ds}.
\ee
\end{proposition}
\begin{proof}
Without loss of the generality, we assume that $f\geq 0$. In order to apply  Lemma \ref{2.1}, let us divide the interval $[0,T]$ into  several intervals $[T_{i-1},T_i]$, $0=T_0<T_1<\cdots<T_k<T_{k+1}=T$, such that 
\begin{align*}
\sup_{x\in\R^d}\E\int_0^{T_i-T_{i-1}}f(T_{i-1}+s,B^x_s)ds \leq \alpha
\end{align*}
holds for some $\alpha<1$. This can be done thanks to Proposition \ref{2.2} and Remark \ref{appre}. Applying Lemma \ref{2.1}, we obtain
\begin{align*}
\sup_{x\in\R^d}\E e^{\int^T_0f(s,B_s^x)ds} 
&=\sup_{x\in\R^d}\E e^{\int^{T_1}_0f(s,B_s^x)ds} \dots e^{\int^T_{T_k}f(s,B_s^x)ds} \\
&=\sup_{x\in\R^d}\E \Big[e^{\int^{T_1}_0f(s,B_s^x)ds} \dots e^{\int^{T_k}_{T_{k-1}}f(s,B_s^x)ds}\E(e^{\int^T_{T_k}f(s,B_s^x)ds}|\mathcal{F}_{T_k})\Big] \\
&= \sup_{x\in\R^d}\E \Big[e^{\int^{T_1}_0f(s,B_s^x)ds} \dots e^{\int^{T_k}_{T_{k-1}}f(s,B_s^x)ds} \E e^{\int^{T-T_k}_0 f(T_k+s,B^y_s)ds}|_{y=B_{T_k}^x}\Big] \\
&\leq \frac{1}{1-\alpha}\sup_{x\in\R^d}\E e^{\int^{T_1}_0f(s,B_s^x)ds} \dots e^{\int^{T_k}_{T_{k-1}}f(s,B_s^x)ds} \\
&\leq \dots \\
&\leq (\frac{1}{1-\alpha})^{k+1}.
\end{align*}
\end{proof}
\begin{remark} \label{re2.4}
In \cite{ff3, krylov}, a similar result is  proved for $f\in L^q([0,T],L^p_x)$ with $\frac{2}{q}+\frac{d}{p}<2$. It is also shown that the quantity \eqref{eq2.1} can be controlled by $\norm{f}_{L^{q}([0,T],L^p_x)}$. At the Orlicz-critical regime, one can control the quantity \eqref{eq2.1} in some weak sense. In fact, thanks to Lemma \ref{2.1} and Proposition \ref{2.2}, there exists a constant $K=K(p,q)$ such that the following  holds: there exists a function $C:\R \rightarrow \R$ such that for any $f$ satisfying $\norm{f}_{L^{q,1}([0,T],L^p_x)}<K$,
\begin{align*}
 \sup_{x\in\R^d}\E e^{\int^T_0f(s,B_s^x)ds} \leq C(\norm{f}_{L^{q,1}([0,T],L^p_x)}).
\end{align*} 
This means that for functions $f$ having sufficiently small $\norm{f}_{L^{q,1}([0,T],L^p_x)}$, the quantity \eqref{eq2.1} can be controlled in terms of $\norm{f}_{L^{q,1}([0,T],L^p_x)}$.
\end{remark}
Now, as an application of the Girsanov theorem, one can  derive the existence of a weak solution to SDE \eqref{SDE} under the Orlicz-critical condition \eqref{condition}:
\begin{theorem}\label{2.4}
Suppose that $b$ satisfies the condition \eqref{condition}. Then, SDE  \eqref{SDE} admits a weak solution. More precisely, we can construct processes $X_t$ and $B_t$ for $0\leq t\leq T$ on some filtered space ($\Omega, \mathcal{F}, \mathcal{F}_t, P$) such that $B_t$ is a standard $\mathcal{F}_t$-Brownian motion and almost surely,
\be 
X_t=x+\int_0^tb(s,X_s)ds+B_t
\ee
holds for all $0\leq t\leq T$.
\end{theorem}
\begin{proof}
Let $X_t$ be a Brownian motion starting from $x$ on the probability space ($\Omega, \mathcal{G}, Q$), equipped with a natural filtration $\mathcal{F}_t$. Then, using  Proposition \ref{2.3}, one can conclude that
\begin{align*}
\alpha_t=\exp\Big[\int_0^t b(s,X_s)dX_s-\frac{1}{2}\int_0^t|b(s,X_s)|^2ds\Big]
\end{align*}
is a $Q$-martingale since the Novikov condition is satisfied.
Thus, a process defined by
\begin{align*}
B_t=X_t-\int_0^t b(s,X_s)ds-x
\end{align*}
is a $\mathcal{F}_t$-Brownian motion starting from the origin with respect to the new probability measure $dP(w)=\alpha_T(w)dQ(w)$ on $\mathcal{F}_T$ due to the Girsanov theorem. 
\end{proof}

\subsection{Associated PDE results} \label{section 3.2}
In this section, we study the following  Kolmogorov PDE:
\be \label{Kolmogorov PDE}
\begin{cases}
u_t - \frac{1}{2} \Delta u+b\cdot \nabla u+f=0, \quad 0\leq t\leq T, \\
u(0,x)=0,
\end{cases}
\ee
 for singular functions $b$ and $f$ in the Orlicz-critical space 
\eqref{condition}. 
This PDE \eqref{Kolmogorov PDE} provides  a key ingredient to prove the strong uniqueness of SDE \eqref{SDE}. 

The PDE \eqref{Kolmogorov PDE} has been extensively studied  when singular coefficients $b$ and $f$ belong to the subcritical Lebesgue space \eqref{KR} (see for example \cite{ff3,krylov,zhang3}). On the other hand, a theory of PDE \eqref{Kolmogorov PDE} with critical coefficients has not been well-established due to the lack of nice  embedding  properties for the mixed-norm parabolic Sobolev spaces at the critical regime. In this section, we obtain the parabolic Sobolev embedding properties under the case when a slightly stronger Lorentz integrability condition is imposed on the time variable (see Proposition \ref{2.7}). From this, we establish the well-posedness result of PDE \eqref{Kolmogorov PDE} with singular coefficients in the Orlicz-critical  spaces \eqref{condition}, and then obtain a priori estimate of a solution.

For $1<p,q<\infty$, and $S\leq T$, let us define a function space $X^{q,p}([S,T])$ to be a collection of functions satisfying
\begin{align*}
u, u_t, \nabla u, \nabla^2 u\in L^{q,1}([S,T],L^p_x).
\end{align*}
Note that derivatives are interpreted as a distribution sense. Its norm is defined by
\begin{multline*}
\norm{u}_{X^{q,p}([S,T])} \\
:=\norm{u}_{L^{q,1}([S,T],L^p_x)}+\norm{u_t}_{L^{q,1}([S,T],L^p_x)}+\norm{\nabla u}_{L^{q,1}([S,T],L^p_x)}+\norm{\nabla^2 u}_{L^{q,1}([S,T],L^p_x)}.
\end{multline*} 
One can easily check that $X^{q,p}([S,T])$ is a quasi-Banach space. The main result of this section is the following theorem, which establishes the well-posedness of PDE \eqref{Kolmogorov PDE} and a priori estimate of a solution: 
\begin{theorem} \label{2.5}
Assume that $b$ satisfies \eqref{condition}. 
Then, there exists $T_0\leq T$ satisfying the following properties: for any $f \in L^{q,1}([0,T_0],L^p_x)$,  there exists a unique solution $u\in X^{q,p}([0,T_0])$ to \eqref{Kolmogorov PDE} for $0\leq t\leq  T_0$, and the estimate 
\be \label{a priori}
\norm{u}_{X^{q,p}([0,T_0])}\leq C\norm{f}_{L^{q,1}([0,T_0],L^p_x)}
\ee
holds for some constant $C$ depending only on $\norm{b}_{L^{q,1}([0,T_0],L^p_x)}$.
\end{theorem}
The first step to establish this theorem is to obtain an a priori estimate for the $L^{q,1}([0,T],L^p_x)$-norm of the following heat equation:
\be \label{heat}
\begin{cases}
u_t-\frac{1}{2}\Delta u=f, \quad 0\leq t\leq T, \\
u_0=0.
\end{cases}
\ee 

\begin{proposition} \label{2.6}
For any $p,q\in (1,\infty)$ and $f\in L^{q,1}([0,T],L^p_x)$, there exists a unique solution $u\in X^{q,p}([0,T])$ to PDE  \eqref{heat}. Also, there exists  some constant $C=C(p,q)$ independent of $T$ such that for any $f\in L^{q,1}([0,T],L^p_x)$, 
\be  \label{estimate}
\norm{\nabla^2 u}_{L^{q,1}([0,T],L^p_x)} \leq C\norm{f}_{L^{q,1}([0,T],L^p_x)},
\ee 
\be \label{estimate1}
\norm{u}_{X^{q,p}([0,T])} \leq C\max\{1,T\} \norm{f}_{L^{q,1}([0,T],L^p_x)}.
\ee 

\end{proposition}
\begin{proof}
Let us first prove the estimate \eqref{estimate}. For $f\in C^\infty_c([0,T]\times \R^d)$, let us define $u(t)=\int_0^t T_{t-s}f(s)ds$, where $T_t$ denotes the semigroup generated by $\frac{1}{2}\Delta$. Obviously, $u$ is a classical solution to the heat equation \eqref{heat}. 
According to \cite[Theorem 1.2]{krylov2}, for any $p,q\in (1,\infty)$, there exists some constant $C=C(p,q)$ independent of $T$ such that for any  $f\in L^q([0,T],L^p_x)$,
\begin{align*}
\norm{\nabla^2 u}_{L^{q}([0,T],L^p_x)} \leq C\norm{f}_{L^{q}([0,T],L^p_x)}.
\end{align*} 
 Since $L^{q,1}_t(L^p_x)$ can be realized as a real interpolation space of two mixed-norm Lebesgue spaces: for $0<\theta<1$ satisfying $\frac{1}{q}=\frac{1-\theta}{q_1}+\frac{\theta}{q_2}$,
\begin{align*}
[L^{q_1}([0,T],L^p_x), L^{q_2}([0,T],L^p_x)]_{\theta, 1}=L^{q,1}([0,T],L^p_x),
\end{align*}
 we obtain the estimate \eqref{estimate} (see \cite{interpolation} for the details of interpolation spaces). Also, using the equation \eqref{heat} and the estimate \eqref{estimate}, for some constant $C_1=C_1(p,q)$ independent of $T$,
\begin{align*}
\norm{u_t}_{L^{q,1}([0,T],L^p_x)} \leq C_1\norm{f}_{L^{q,1}([0,T],L^p_x)}.
\end{align*}
 Using the Minkowski's integral inequality, H\"older's inequality and the trivial inequality $u(t,x)\leq \int_0^T |u_t(s,x)|ds$, it follows that for some constant $C_2=C_2(p,q)$ independent of $T$,
\begin{align*}
\norm{u}_{L^{q,1}([0,T],L^p_x)}\leq C_2T\norm{f}_{L^{q,1}([0,T],L^p_x)}.
\end{align*} 
 Furthermore, using the interpolation inequality $\norm{\nabla u}_{L^p_x} \lesssim \norm{u}_{L^p_x}+\norm{\nabla^2 u}_{L^p_x}$ and the aforementioned results, we readily obtain \eqref{estimate1}. 
 
 The existence of a solution $u\in X^{q,p}([0,T])$ to the heat equation \eqref{heat} can be established via a standard approximation argument and the  estimate \eqref{estimate1}. Uniqueness immediately follows from the estimate \eqref{estimate1}.
\end{proof}

In order to obtain an a priori estimate \eqref{a priori} for the PDE \eqref{Kolmogorov PDE} using the result in Proposition \ref{2.6}, we need to handle the first order term $\norm{b\cdot \nabla u}_{L^{q,1}([0,T],L^p_x)}$. Since $b\in L^{q,1}([0,T],L^p_x)$, this term can be controlled  once we are able to control $\norm{\nabla u}_{L^\infty([0,T]\times \R^d)}$. The embedding theorems for the mixed-norm parabolic Sobolev spaces are obtained in \cite[Lemma 10.2]{krylov}: $\nabla u$ is bounded and H\"older continuous in $(t,x)$ provided that
\begin{align*}
u_t, \nabla^2 u\in L^q([0,T],L^p_x)
\end{align*} 
for $1<p,q<\infty$ satisfying the subcritical condition $\frac{2}{q}+\frac{d}{p}<1$. However, in general, $\nabla u$  may not be bounded under the critical condition $\frac{2}{q}+\frac{d}{p}=1$: recall that the Sobolev embedding $W^{1,d}(\R^d) \hookrightarrow L^\infty(\R^d)$ fails at the critical regime.  Remarkably,  when a slightly stronger Lorentz integrability condition is imposed on the time variable, the boundedness of $\nabla u$ can be established at the critical regime $\frac{2}{q}+\frac{d}{p}=1$:
\bepro \label{2.7}
Suppose that $u\in X^{q,p}([0,T])$ with  $u(0)=0$, and the exponents $1<p, q<\infty$ satisfy the condition:
\begin{align*}
\frac{2}{q}+\frac{d}{p}=1.
\end{align*} 
Then $\nabla u\in L^\infty([0,T]\times \R^d)$. Also, there exists some constant $C=C(p,q)$ independent of $T$  such that for any $u\in X^{q,p}([0,T])$,
\be \label{embedding}
\norm{\nabla u}_{L^\infty([0,T]\times \R^d)} \leq C (\norm{u_t}_{L^{q,1}([0,T],L^p_x)}+\norm{\nabla^2 u}_{L^{q,1}([0,T],L^p_x)}).
\ee
\enpro
\begin{proof}
Let us define $f:=u_t-\Delta u$. One can represent $\nabla u$ in terms of the heat kernel:
\begin{align*}
\nabla u(t,x)= \int_0^t \int_{\R^d} \nabla (\frac{1}{s^{d/2}}e^{-|y|^2/4s})\cdot f(t-s,x-y)dyds.
\end{align*}
If we denote $p',q'$ by the conjugate exponents of $p,q$, respectively, then according to  Proposition \ref{A.2}, we have
\begin{align*}
\nabla (\frac{1}{t^{d/2}}e^{-|x|^2/4t})\in L^{q',\infty}([0,T],L^{p'}_x).
\end{align*}
Thus, using the O'Neil's inequality for the mixed-norm Lorentz spaces (Proposition \ref{A.3}), 
\begin{align*}
\norm{\nabla u}_{L^\infty([0,T]\times \R^d)} &\leq C \norm{\nabla (\frac{1}{t^{d/2}}e^{-|x|^2/4t})}_{L^{q',\infty}([0,T],L^{p'}_x)}\norm{f}_{L^{q,1}([0,T],L^p_x)}\\
 &\leq  C(p,q)(\norm{u_t}_{L^{q,1}([0,T],L^p_x)}+\norm{\nabla^2 u}_{L^{q,1}([0,T],L^p_x)}).
\end{align*}
Note that the estimate Proposition \ref{A.3} is global in time, whereas the above inequality is integrated only over $[0,T]$. This subtle problem can be easily overcome by extending two functions $g(s,y)=\nabla (\frac{1}{s^{d/2}}e^{-|y|^2/4s})$ and $f(s,y)$ to the whole real line by setting $f,g=0$ outside $ [0,T]$. 
\end{proof}
\begin{remark} \label{re2.9}
In \cite{parabolic},  parabolic Riesz potentials are studied in the context of the mixed-norm spaces. If we denote $p(t,x)$ by the standard heat kernel, then the operator defined by 
\begin{align*}
p* f(t,x):= \int_0^\infty \int_{\R^d} p(s,y)f(t-s,x-y)dyds
\end{align*}
is bounded from $L^{q_1}(\R,L^{p_1}_x)$ to $L^{q_2}(\R,L^{p_2}_x)$ for  $1\leq p_1<p_2<\infty$ and $1\leq q_1<q_2<\infty$ satisfying $1=\frac{d}{2}(\frac{1}{p_1}-\frac{1}{p_2})+(\frac{1}{q_1}-\frac{1}{q_2})$. Note that this result does not include the endpoint case $p_2=q_2=\infty$. However, one can cover the endpoint case Proposition \ref{2.7}, at the price that a slightly stronger Lorentz norm shows up in the right hand side of  \eqref{embedding}.
\end{remark}
Now, we are ready to study the Kolmogorov PDE \eqref{Kolmogorov PDE}. 
\begin{proof}[Proof of Theorem \ref{2.5}]
We use a fixed point theorem for the quasi-Banach spaces (see Proposition \ref{A.4}) to prove the existence of a solution. For $u \in X^{q,p}([0,T])$, we have $\nabla u\in L^\infty([0,T]\times \R^d)$ according to Proposition \ref{2.7}. Therefore, for $b,f\in L^{q,1}([0,T],L^p_x)$, we have $f+b\cdot \nabla u\in L^{q,1}([0,T],L^p_x)$. Using Proposition \ref{2.6}, let us define $w=F(u)\in X^{q,p}([0,T])$ to be a unique solution of the following PDE:
\begin{align*}
\begin{cases}
w_t-\frac{1}{2}\Delta w=-(f+b\cdot \nabla u), \quad 0\leq t\leq T, \\
w(0,x)=0.
\end{cases}
\end{align*} 
Using the estimates \eqref{estimate1} and \eqref{embedding}, for some constants $C,C_1$ independent of $T$,
\begin{align*}
\norm{F(u_1)-F(u_2)}_{X^{q,p}([0,T])}  &\leq C\max\{1,T\}\norm{b\cdot \nabla(u_1-u_2)}_{L^{q,1}([0,T],L^p_x)} \\
&\leq C\max\{1,T\}\norm{b}_{L^{q,1}([0,T],L^p_x)}\cdot \norm{\nabla (u_1-u_2)}_{L^\infty([0,T]\times \R^d)} \\
&\leq C_1\max\{1,T\}\norm{b}_{L^{q,1}([0,T],L^p_x)}\cdot \norm{(u_1-u_2)}_{X^{q,p}([0,T])}.
\end{align*}
Let us denote  $c=c(q,1)>1$ by a constant from \eqref{quasi}, and choose a sufficiently small $T_0$ satisfying
\begin{align*}
\norm{b}_{L^{q,1}([0,T_0],L^p_x)}<\frac{1}{2cC_1\max\{1,T_0\}}
\end{align*}
(see Remark \ref{appre} for its validity). Then, a map $F:X^{q,p}([0,T_0])\rightarrow X^{q,p}([0,T_0])$ satisfies
\begin{align*}
|F(x)-F(y)|<\frac{1}{2c}|x-y|.
\end{align*} 
Therefore, applying a fixed point theorem for the quasi-Banach spaces (see Proposition \ref{A.4}), there exists $u\in X^{q,p}([0,T_0])$ satisfying  PDE \eqref{Kolmogorov PDE} for $0\leq t\leq T_0$. 

Now, let us prove the estimate \eqref{a priori}. Using \eqref{estimate1} and \eqref{quasi}, for some constants $C,C_1$,
\begin{align*}
\norm{u}_{X^{q,p}([0,T_0])} &\leq C\max\{1,T_0\}\norm{f+b\cdot \nabla u}_{L^{q,1}([0,T_0],L^p_x)} \\
&\leq C_1\max\{1,T_0\}(\norm{f}_{L^{q,1}([0,T_0],L^p_x)}+\norm{b}_{L^{q,1}([0,T_0],L^p_x)}\norm{u}_{X^{q,p}([0,T_0])}).
\end{align*}
Therefore, for sufficiently small $T_0$ satisfying
\be \label{12}
\norm{b}_{L^{q,1}([0,T_0],L^p_x)}<\frac{1}{C_1\max\{1,T_0\}},
\ee
 we obtain the estimate \eqref{a priori}. Note that a constant $C$ in \eqref{a priori} can be chosen depending only on $\norm{b}_{L^{q,1}([0,T_0],L^p_x)}$. 
\end{proof}
\begin{remark} \label{stable}
From the proof of Theorem \ref{2.5}, one can check that for any $b$ with sufficiently small $\norm{b}_{L^{q,1}([0,T],L^p_x)}$, there exists a unique solution $u$ to PDE \eqref{Kolmogorov PDE} for $0\leq t\leq T$ satisfying:
\begin{align*}
\norm{u}_{X^{q,p}([0,T])}\leq C(\norm{b}_{L^{q,1}([0,T],L^p_x)},p,q)\norm{f}_{L^{q,1}([0,T],L^p_x)}.
\end{align*} For these $b$'s, one can easily derive a stability property of PDE \eqref{Kolmogorov PDE}. More precisely, there exist a constant $C_0$ depending on $T$ satisfying the following statement: for any $b_i$ and $f_i$, $i=1,2$, satisfying 
\begin{align*}
\norm{f_i}_{L^{q,1}([0,T],L^p_x)}, \norm{b_i}_{L^{q,1}([0,T],L^p_x)}<C_0,
\end{align*}
 define $u_i$ to be a solution to PDE \eqref{Kolmogorov PDE} with $b_i$ and $f_i$ in place of $b$ and $f$, respectively. Then,   for some constant $\bar{C}>1$ depending on $C_0$,
\begin{align} 
\norm{u_1-u_2}_{X^{q,p}([0,T])}, \norm{u_1-u_2}_{L^\infty([0,T]\times \R^d)}, \norm{\nabla(u_1-u_2)}_{L^\infty([0,T]\times \R^d)} \nonumber \\
 \leq \frac{\bar{C}}{2}(\norm{b_1-b_2}_{L^{q,1}([0,T],L^p_x)}+\norm{f_1-f_2}_{L^{q,1}([0,T],L^p_x)}).  \label{stability estimate 1}
\end{align} 
 In particular, when $f_i=b_i$, the RHS of \eqref{stability estimate 1} can be written as $\bar{C}\norm{b_1-b_2}_{L^{q,1}([0,T],L^p_x)}$.
\end{remark}

Assume that $b$ satisfies \eqref{condition}, and $T_0$ is from Theorem \ref{2.5}. According to Theorem \ref{2.5}, there exists a unique solution $\tilde{u}\in X^{q,p}([0,T_0])$ to the following PDE: 
\be \label{eq2.9}
\begin{cases}
u_t + \frac{1}{2} \Delta u+b\cdot \nabla u+b=0, \quad 0\leq t\leq T_0, & \\
u(T_0,x)=0.
\end{cases}
\ee
The following proposition plays an essential role in Section \ref{section 3.3}.
\bepro \label{2.8}
There exists a sufficiently small $T_1$ such that the following holds: if $\tilde{u}$ is a solution to \eqref{eq2.9} with $T_1$ in place of   $T_0$, then there exists a version $u$ of $\tilde{u}$, which is continuous in $(t,x)$, such that $\Phi(t,x):=x+u(t,x)$ satisfies the following conditions: \\
(i) $\Phi(t,\cdot)$ is a $C^1$ diffeomorphism from $\R^d$ to itself for each $0\leq t\leq T_1$.\\
(ii) For each $0\leq t\leq T_1$,
\begin{align*}
\frac{1}{2}\leq \norm{\nabla \Phi(t,\cdot)}_{L^\infty(\R^d)}\leq 2, \quad \frac{1}{2}\leq \norm{\nabla \Phi^{-1}(t,\cdot)}_{L^\infty(\R^d)}\leq 2.
\end{align*} 
\enpro
Here, we say $u_1$ is a version of $u_2$ if $u_1=u_2$ for $(t,x)$-a.e.
\begin{proof}
Let us first prove that there exist a version $u$ of $\tilde{u}$ which is $C^1$ in $x$. Choose a smooth approximation $u_n$ of $\tilde{u}$ in $X^{q,p}([0,T_0])$ norm. Thanks to Proposition \ref{2.7}, 
\begin{align*}
\norm{\nabla (u_n-u_m)}_{L^\infty_{t,x}([0,T_0]\times \R^d)} \leq C \norm{u_n-u_m}_{X^{q,p}([0,T_0])}.
\end{align*}
Therefore, $\nabla u_n$ converge uniformly to some continuous function $w$. Since $u_n$ converge uniformly to some continuous function $u$ which is a version of $\tilde{u}$,  $u$ is differentiable in $x$ and its spatial derivative is $w$. Since $w$ is continuous, $u$ is $C^1$ in $x$.

Now, let us show that for sufficiently small $T_1$, $\nabla \Phi(t,x)$ is non-singular for each $0\leq t\leq T_1$. Note that using the estimates \eqref{estimate1} and \eqref{embedding}, for some constants $C,C_1,C_2$ independent of $T$,
\begin{align*}
\norm{\nabla u}_{L^\infty([0,T]\times \R^d)} &\leq C\norm{u}_{X^{q,p}([0,T])} \leq C_1\max\{1,T\}\norm{b\cdot \nabla u+b}_{L^{q,1}([0,T],L^p_x)} \\
&\leq C_2\max\{1,T\}(\norm{b}_{L^{q,1}([0,T],L^p_x)}\norm{\nabla u}_{L^\infty([0,T]\times \R^d)}+\norm{b}_{L^{q,1}([0,T],L^p_x)}).
\end{align*}
Therefore, if we choose sufficiently small $T_1$ so that $\norm{b}_{L^{q,1}([0,T_1],L^p_x)}$ is small enough, then
\begin{align*}
\norm{\nabla u}_{L^\infty([0,T_1]\times \R^d)} \leq \frac{1}{2}.
\end{align*}
This immediately implies the first inequality in the condition (ii). From this, we obtain the non-singularity of $\nabla \Phi(t,\cdot)$,  and
$\lim_{|x|\rightarrow \infty}|\Phi(t,x)|=\infty$ for each $t\in [0,T_1]$.  Therefore, according to the Hadamard's Lemma (see Proposition \ref{A.7}), $\Phi(t,\cdot)$ is a global diffeomorphism for each $t\in [0,T_1]$, which concludes the proof of the first property.  

The second inequality in (ii)  follows from the identity
\begin{align*}
\nabla \Phi^{-1}(t,x)=[\nabla \Phi(t,\Phi^{-1}(t,x))]^{-1}=[I+\nabla u(t,\Phi^{-1}(t,x))]^{-1},
\end{align*} 
and the fact $\sup_{t\in [0,T_1]}\norm{\nabla u}_{L^\infty(\R^d)}\leq \frac{1}{2}$.
\end{proof}

\begin{remark} \label{re2.12}
In \cite{ff3,f}, authors considered the following PDE with a potential $\lambda u$ ($\lambda>0$):
\begin{align*}
u_t+\frac{1}{2}\Delta u-b\cdot \nabla u -\lambda u=b
\end{align*}
in order to obtain a global bijectivity of the map $\Phi(t,\cdot)$. They proved that for sufficiently large $\lambda$, $\norm{\nabla u}_{L^\infty([0,T]\times \R^d)}<\frac{1}{2}$. However, this method is not applicable in our case due to the critical nature of the exponents $p$ and $q$. Instead, we accomplished this by taking the time  $T$ sufficiently small.
\end{remark}
From now on, we use the notations $u(t,x)$, $\Phi(t,x)$, and $T_1$ from  Proposition \ref{2.8}. 
\subsection{Uniqueness of a strong solution to SDE} \label{section 3.3}
In this section, we prove the uniqueness of a  strong solution to SDE \eqref{SDE} up to time $T_1$. The following proposition claims that a strong solution to \eqref{SDE} yields a new strong solution to the auxiliary SDE which contains no drift terms. It is called the Zvonkin's transformation method  \cite{zvo}. 
\bepro \label{2.9}
Suppose that $b$ satisfies \eqref{condition}, and $X_t$ is a strong solution to SDE  \eqref{SDE} up to time $T_1$. Then, $Y_t$ defined by $Y_t=\Phi(t,X_t)$ is a strong solution to the following SDE:
\be\label{conjugated}
\begin{cases}
dY_t=\tilde{\sigma}(t,Y_t)dB_t, \quad 0\leq t\leq T_1, \\
Y_0=\Phi(0,x)=y,
\end{cases}
\ee
for  $\tilde{\sigma}$ defined by
\be
\tilde{\sigma}(t,x)=I+\nabla u(t,\Phi^{-1}(t,x)).
\ee
\enpro
\begin{proof}
One can check that the standard It\^o's formula 
\begin{align*}
f(t,X_t)-f(0,X_0)=\int_0^t (f_t+b\nabla f+\frac{1}{2}\Delta f)(s,X_s)ds + \int_0^t \nabla f (s,X_s)dB_s
\end{align*} 
holds for any functions $f \in X^{q,p}([0,T])$ with $p,q$ satisfying $\frac{2}{q}+\frac{d}{p}=1$.
In fact, the proof in \cite[Theorem 3.7]{krylov} applies to our case without any changes. Thus, applying It\^o's formula to a function $u$, we have
\begin{align*}
u(t,X_t)&=u(0,X_0)+\int_0^t (u_t+b\cdot \nabla u+\frac{1}{2}\Delta u)(s,X_s)ds+\int_0^t \nabla u(s,X_s)dB_s \\
&= u(0,X_0)-\int_0^t b(s,X_s)ds+\int_0^t \nabla u(s,X_s)dB_s \\
&=u(0,X_0)-X_t+X_0+B_t+\int_0^t \nabla u(s,X_s)dB_s.
\end{align*}
Therefore, we obtain
\begin{align*}
Y_t-Y_0=\Phi(t,X_t)-\Phi(t,X_0)&=\int_0^t \nabla u(s,X_s)dB_s+B_t = \int_0^t \nabla u(s,\Phi^{-1}(s,Y_s))dB_s+B_t.
\end{align*}
\end{proof}
Let us call SDE \eqref{conjugated} by a \it{conjugated SDE}. Before proving the strong uniqueness of SDE \eqref{SDE}, we prove the following two lemmas which will be used frequently. 
\begin{lemma} \label{2.10}
For any $\lambda_1,\lambda_2\in\R$ and $b$ satisfying the condition \eqref{condition}, 
\be  \label{eq2.12}
\sup_x \E\exp\Big[\lambda_1\int_0^T b(s,B^x_s)dB^x_s+\lambda_2 \int_0^T b^2(s,B^x_s)ds\Big] <\infty.
\ee
\end{lemma}
\begin{proof}
If we denote $\mathscr{E}(M)_t$ by a Dol\'eans-Dade exponential of the martingale $M_t$, then by H\"older's inequality,
\begin{align}
&\E\exp\Big[\lambda_1\int_0^T b(s,B^x_s)dB^x_s+\lambda_2 \int_0^T b^2(s,B^x_s)ds\Big] \nonumber \\
&\leq \bigg[\E \mathscr{E}\Big[\int_0^T 2\lambda_1b(s,B^x_s)dB^x_s\Big]\bigg]^{1/2}\bigg[\E \exp\Big[(\lambda_2+\lambda_1^2)\int_0^Tb^2(s,B^x_s)ds\Big]\bigg]^{1/2}. \label{zz}
\end{align}
 Since $b\in L^{q,1}([0,T],L^p_x)$, it follows that $b^2\in L^{q/2,1/2}([0,T], L^{p/2}_x)$. Letting $\tilde{q}=\frac{q}{2}$ and $\tilde{p}=\frac{p}{2}$, we have $b^2\in L^{\tilde{q},1}([0,T],L^{\tilde{p}}_x)$ with $\frac{2}{\tilde{q}}+\frac{d}{\tilde{p}}=2$. Therefore, the second term of \eqref{zz} is finite according to Proposition \ref{2.3}. The first term of \eqref{zz} is equal to 1 since the Novikov's condition is satisfied. 
\end{proof}
\belemma \label{2.11}
Let $X_t$ be a solution to SDE \eqref{SDE} with $b$ satisfying the condition \eqref{condition}. Then, for arbitrary $\lambda_1,\lambda_2\in\R$ and $f\in L^{q,1}([0,T],L^p_x)$,
\be \label{eq2.13}
\sup_x \E \exp\Big[\lambda_1 \int_0^T f(s,X_s)dB^x_s+\lambda_2 \int_0^T f^2(s,X_s)ds\Big]<\infty.
\ee
\enlemma
\begin{proof}
By Girsanov formula, LHS of \eqref{eq2.13} equals to
\begin{multline*}
\sup_x \E \bigg[\exp\Big[\lambda_1 \int_0^T f(s,B^x_s)dB^x_s+\lambda_2 \int_0^T f^2(s,B^x_s)ds\Big]\cdot  \\ \exp\Big[\int_0^T b(s,B^x_s)dB^x_s-\frac{1}{2}\int_0^Tb^2(s,B^x_s)ds\Big]\bigg].
\end{multline*}
Since both $b$ and $f$ belong to $L^{q,1}([0,T],L^p_x)$ with $\frac{2}{q}+\frac{d}{p}=1$, H\"older's inequality and Lemma \ref{2.10} conclude the proof.
\end{proof} 
\begin{remark} \label{re2.16}
It is proved in  \cite{ff3,f} that under the subcritical condition \eqref{KR}, quantities \eqref{eq2.12} and \eqref{eq2.13} can be controlled by $\norm{b}_{L^{q}([0,T], L^p_x)}$. At the Orlicz-critical regime \eqref{condition}, these quantities can be controlled by $\norm{b}_{L^{q,1}([0,T], L^p_x)}$ in some weak sense. In fact, by applying Lemma \ref{2.1} and Proposition \ref{2.2} to  Lemma \ref{2.10} and \ref{2.11}, one can show that there exists a constant $K=K(p,q,\lambda_1, \lambda_2)$ and functions $C_1,C_2:\R \rightarrow \R$ such that the following holds: for any $f$ and $b$ satisfying
\begin{align*}
\norm{f}_{L^{q,1}([0,T],L^p_x)},\norm{b}_{L^{q,1}([0,T],L^p_x)}<K,
\end{align*} 
we have
\begin{align*}
\sup_x \E\exp\Big[\lambda_1\int_0^T b(s,B^x_s)dB^x_s+\lambda_2 \int_0^T b^2(s,B^x_s)ds\Big] \leq C_1(K), 
\end{align*}  
\be \label{end1}
\sup_x \E \exp\Big[\lambda_1 \int_0^T f(s,X_s)dB^x_s+\lambda_2 \int_0^T f^2(s,X_s)ds\Big]\leq C_2(K).
\ee  
If we denote $X^{\mu}_t$ by a solution to SDE \eqref{SDE} with the initial distribution $\mu$, then \eqref{end1} implies that
\begin{align} \label{end2}
\sup_{\mu} \E \exp\Big[\lambda_1 \int_0^T f(s,X^{\mu}_s)dB^x_s+\lambda_2 \int_0^T f^2(s,X^{\mu}_s)ds\Big]\leq C_2(K)
\end{align} 
($\sup$ takes over all of the probability measures on $\R^d$). This is because if we denote $P_x$ by a law of $\{X_t \ |\ 0\leq t\leq T\}$ which is a solution of \eqref{SDE} starting from $x$, then $P_{\mu}=\int P_x d\mu(x)$ is a law of $\{X^{\mu}_t \ |\ 0\leq t\leq T\}$.
 
Also, by letting $\lambda_1=0$ and $\lambda_2=1$ in Lemma \ref{2.11} and using the inequality $1+x \leq e^x$, one can conclude that  there exists a function $C:\R \rightarrow \R$ such that for any $f$ and $b$ satisfying
\begin{align*}
\norm{f}_{L^{q,1}([0,T],L^p_x)}, \norm{b}_{L^{q,1}([0,T],L^p_x)}<K(p,q,0,1),
\end{align*}
we have 
\begin{align*}
\sup_x \E \int_0^T f^2(s,X_s) ds <C(K).
\end{align*}
\end{remark}
Now, we are ready to prove the strong uniqueness of SDE \eqref{SDE} under the condition  \eqref{condition} using Lemma \ref{2.10} and \ref{2.11}. Proof follows the argument  in \cite[Theorem 4.1]{ff3}.
\bepro  \label{2.17}
A strong solution to SDE \eqref{SDE} is unique up to $T_1$.
\enpro
\begin{proof}
Let $X^1_t$ and $X^2_t$ be strong solutions to SDE \eqref{SDE} starting from $x^1$ and $x^2$, respectively. According to Proposition \ref{2.9}, if we define $Y^i_t=\Phi(t,X^i_t)$, then $Y^i_t$ is a solution to the conjugated SDE \eqref{conjugated} starting from $y^i=\Phi(0,x^i)$, respectively. Thus, we have
\be 
d(Y^1_s-Y^2_s)=[\tilde{\sigma}(s,Y^1_s)-\tilde{\sigma}(s,Y^2_s)]dB_s.
\ee
For any $r\in (1,\infty)$, using  the It\^o's formula, 
\begin{align*}
d&|Y^1_s-Y^2_s|^r \\
&=\frac{r(r-1)}{2} \text{Trace}\big([\tilde{\sigma}(s,Y^1_s)-\tilde{\sigma}(s,Y^2_s)][\tilde{\sigma}(s,Y^1_s)-\tilde{\sigma}(s,Y^2_s)]^T\big)|Y^1_s-Y^2_s|^{r-2}ds+dM_s \\
&\leq \frac{r(r-1)}{2} |\tilde{\sigma}(s,Y^1_s)-\tilde{\sigma}(s,Y^2_s)|^2|Y^1_s-Y^2_s|^{r-2} ds + dM_s \\
&=|Y^1_s-Y^2_s|^rdA_s+dM_s
\end{align*}
for some martingale $M_s$ with zero mean (the martingale property can be checked as in \cite[Theorem 5.6]{f4}).
Here, we introduced an auxiliary process $A_t$ ($0\leq t\leq T_1$) satisfying
\be  \label{At}
\frac{r(r-1)}{2} \int^t_0 |\tilde{\sigma}(s,Y^1_s)-\tilde{\sigma}(s,Y^2_s)|^2ds=\int^t_0|Y^1_s-Y^2_s|^2dA_s,
\ee
and for any $c>0$,
\be 
\E e^{cA_t}<\infty
\ee
(with the aid of Lemma \ref{2.10} and \ref{2.11},  the proof of \cite[Lemma 4.5]{ff3}  applies to our case without any changes).
Thus, applying the  product rule,
\begin{align*}
d(e^{-A_s}|Y^1_s-Y^2_s|^r)&=-e^{-A_s}|Y^1_s-Y^2_s|^r dA_s+e^{-A_s}d|Y^1_s-Y^2_s|^r \leq e^{-A_s}dM_s.
\end{align*}
Integrating this inequality in time and then taking the expectation,  we have
\begin{align*}
\E [e^{-A_t}|Y^1_t-Y^2_t|^r] \leq |y^1-y^2|^r.
\end{align*}
Therefore, using the H\"older's inequality,
\begin{align*}
\E |Y^1_t-Y^2_t|^{r/2} &= \E e^{\frac{-A_t}{2}}|Y^1_t-Y^2_t|^{r/2} e^{\frac{A_t}{2}} \\
&\leq [\E e^{-A_t}|Y^1_t-Y^2_t|^{r} ]^{1/2}[\E e^{A_t}]^{1/2} \leq |y^1-y^2|^r[\E e^{A_t}]^{1/2},
\end{align*}
which implies that for each $t\in [0,T_1]$,
\be  \label{H\"older in x}
\E |Y^1_t-Y^2_t|^{r/2} \leq C|y^1-y^2|^{r/2}.
\ee
In particular, when $x^1=x^2$, we have $\E|Y^1_t-Y^2_t|^{r/2}=0$. Since trajectories are continuous and $\Phi(t,\cdot)$ is bijective, we obtain the strong uniqueness of SDE \eqref{SDE}.
\end{proof}
\begin{theorem} \label{2.13}
Existence and uniqueness of a strong solution to SDE \eqref{SDE} holds up to time $T_1$.
\end{theorem}
\begin{proof}
Note that we proved the weak existence in Theorem \ref{2.4} and the strong uniqueness in Proposition \ref{2.17}. Therefore, according to the Yamabe-Watanabe principle \cite{YW,YW1}, we  obtain the existence and uniqueness of a strong solution to SDE \eqref{SDE} up to time $T_1$.
\end{proof}

In the next section, we  construct a strong solution to SDE \eqref{SDE} up to time $T$ as an application of Theorem \ref{2.13}.

\section{Sobolev regularity of a solution} \label{section 4}
In this section, we study the regularity and stability properties of a solution to SDE \eqref{SDE} under the condition \eqref{condition}. In Section \ref{section 4.1}, we construct a  stochastic flow to  SDE \eqref{SDE}. Section \ref{section 4.2} is devoted to study the Sobolev regularity and stability of the stochastic flow. 

\subsection{Construction of the stochastic flow} \label{section 4.1}
Let us first define a stochastic flow.
\begin{definition}
(Stochastic flow). A map $(s,t,x,w) \rightarrow \phi(s,t,x)(w)$, $0\leq s\leq t\leq T$ is called a \it{stochastic flow} associated to the stochastic differential equation \eqref{SDE} on the filtered space with a Brownian motion $(\Omega, \mathcal{F}, \mathcal{F}_t, P, B_t)$ provided that it satisfies: \\
(i) For any $x\in \R^d$ and $0\leq s\leq T$, the process $X^s_{t,x}=\phi(s,t,x)$ for $s\leq t\leq T$ is a $\mathcal{F}_{s,t}$-adapted solution to SDE \eqref{SDE}. Here, $\mathcal{F}_{s,t}:=\sigma(B_u-B_r|s\leq r\leq u\leq t)$. \\
(ii) $w$-almost surely, $\phi(s,t,x)=\phi(u,t,\phi(s,u,x))$ holds for any $0 \leq s \leq u\leq t\leq T$ and $x\in \R^d$. 
\end{definition}
 We refer to \cite{ku1} for the classical theory of stochastic flows. This classical theory has been extended to a large class of SDEs with singular coefficients. For instance, Flandoli et al. \cite{gub} constructed a regular stochastic flow when the SDE with additive noise possess a low H\"older regularity of drift. 

In this section, we prove that a stochastic flow  associated with SDE \eqref{SDE} exists under the Orlicz-critical condition   \eqref{condition}.
The following theorem, combined with Proposition \ref{2.17}, immediately implies Theorem \ref{theorem1} and the first part of Theorem \ref{theorem2}.
\begin{theorem} \label{3.2}
 There exists a stochastic flow $\phi$ to \eqref{SDE} up to time $T$.
\end{theorem}
The main ingredient to prove Theorem \ref{3.2} is the Kolmogorov regularity theorem. Thanks to Proposition \ref{2.9} and Theorem \ref{2.13}, there exists a strong solution $Y^y_t$, $0\leq t\leq T_1$, to \eqref{conjugated}. We first prove the H\"older regularity of $Y^y_t$ using the method in \cite{f}.
\bepro
There exists some constant $C$ such that for any $1\leq r< \infty$, $0\leq t<s\leq T_1$, and $x,y\in \R^d$, 
\begin{align*}
\E|Y_t^x-Y_s^x|^r \leq C|t-s|^{\frac{r}{2}}, \quad \E|Y_t^x-Y_t^y|^r \leq C |x-y|^r.
\end{align*}

\enpro
\begin{proof}
Let us prove the first inequality. Applying the Burkholder-Davis-Gundy inequaltiy and using the fact that $\norm{\nabla u}_{L^\infty([0,T_1]\times \R^d)}$ is finite, one can conclude that
\begin{align*}
\E |Y_t^y-Y_s^y|^r&= \E |\int_s^t (I+\nabla u(\sigma,\Phi^{-1}(r,Y^x_\sigma)))dB_\sigma|^r \\
&\leq C\E |\int_s^t |I+\nabla u(\sigma,\Phi^{-1}(\sigma,Y^x_\sigma))|^2 d\sigma|^{\frac{r}{2}} \leq C|t-s|^{\frac{r}{2}}.
\end{align*}
We have already obtained the second inequality in \eqref{H\"older in x}. 
\end{proof} 
Now, one can  prove Theorem \ref{3.2} by applying the Kolmogorov's regularity theorem. 
\begin{proof}[Proof of Theorem \ref{3.2}]
Since both $\Phi$ and $\Phi^{-1}$ are  continuous in $(t,x)$, we first prove the same statement for the conjugated SDE  \eqref{conjugated}.
Thanks to the Kolmogorov's regularity theorem, one can construct a stochastic flow $\psi$ associated with SDE \eqref{conjugated} up to time $T_1$, which is a version of $Y^y_t$, satisfying the following property: almost surely,  $\psi(s,\cdot,\cdot)$ is $(\alpha, \beta)$-H\"older continuous for each $0\leq s\leq T_1$ and any $0<\alpha<\frac{1}{2}$, $0<\beta<1$. In order to construct a stochastic flow of SDE \eqref{SDE}, let us define 
\begin{align*}
\phi(s,t,x):=\Phi^{-1}(t,\psi(s,t,\Phi(s,x)))
\end{align*}  for $0\leq s\leq t\leq T_1$. It is obvious that  $\phi$ is a  stochastic flow associated with \eqref{SDE} up to time $T_1$, and almost surely, $\phi(s,\cdot,\cdot)$ is continuous for each $0\leq s\leq T_1$. 

Now, we extend this construction globally up to time $T$. Divide $[0,T]$ into the finite number of intervals $[T_{k-1},T_{k}]$, $1\leq k\leq N$, such that the stochastic flow $\phi$ of SDE \eqref{SDE} on each $[T_{k-1},T_{k}]$ can be constructed. More precisely, we take a  sufficiently small interval $[T_{k-1},T_{k}]$ such that the following property holds: if $u^k$ is a solution to PDE
\be \label{3.1}
\begin{cases}
u^k_t + \frac{1}{2}\Delta u^k+b\cdot \nabla u^k+b=0, \quad T_{k-1}\leq t\leq T_k, \\
u^k(T_k,x)=0,
\end{cases}
\ee
then $u^k$  satisfies the conditions in  Proposition \ref{2.8}. In other words,
 $\Phi^k(t,x)=x+u^k(t,x)$ is a global diffeomorphism for each $T_{k-1}\leq t\leq T_k$ and 
\begin{align} \label{infinity}
\frac{1}{2}< \norm{\nabla \Phi^k(t,x)}_{L^\infty([T_{k-1},T_k]\times \R^d)}, \norm{\nabla^{-1} \Phi^k(t,x)}_{L^\infty([T_{k-1},T_k]\times \R^d)}<2.
\end{align} 
 Repeating the arguments mentioned before, one can construct a stochastic flow $\phi(s,t,x)$ associated with SDE \eqref{SDE} for $T_{k-1}\leq s\leq t\leq T_k$. Then, we can glue them together as follows: for each $0\leq s\leq t\leq T$, choose the indices $i$ and $j$ satisfying
\begin{align*}
T_{i-1} \leq s<T_{i}<\cdots<T_{j}<t\leq T_{j+1},
\end{align*} 
and then define
\be 
\phi(s,t,\cdot)=\phi(T_{j},t,\cdot) \circ \phi(T_{j-1},T_j,\cdot) \circ  \dots \circ \phi(s,T_i,\cdot).
\ee
Here, composition happens in the spatial variable. It is obvious that $\phi$ satisfies the properties of the stochastic flow. 
\end{proof} 
\subsection{Sobolev regularity and stability of the stochastic flow} \label{section 4.2}
In the previous section, we constructed the stochastic flow $\phi$ associated with SDE \eqref{SDE}. In this section, we show that $\phi$ is almost surely weakly differentiable in the spatial variable. More precisely, we prove the following theorem, which is a restatement of the second part of Theorem \ref{theorem2}:
\begin{theorem}\label{3.4}
 For each $r\in [1,\infty)$ and $t\in [0,T]$, $\phi(0,t,\cdot)$ is weakly differentiable almost surely and its weak derivative satisfies
\be 
\sup_{x\in \R^d} \E |\nabla \phi(0,t,\cdot)|^r <\infty.
\ee
\end{theorem}
This theorem is proved in several steps. First of all, we approximate $b$ by suitable smooth drifts $b_n$, and then show the weak compactness of  stochastic flows $\phi_n$ associated with smooth drifts $b_n$. We also obtain the convergence of stochastic flows $\phi_n$ to $\phi$ in a suitable topology. Combining these results,  one can conclude the proof of Theorem \ref{3.4}. 

Recall that we first constructed a stochastic flow on each small time interval, and then we obtained a global stochastic flow by gluing together. Due to this nature of the stochastic flow, we need to take a careful approximation to $b$. Let us define $K=K(p,q,0,1)$ and $N=N(d)$ by constants from the Remark \ref{re2.16} and  Proposition $\ref{A.5}$, respectively. Also, we denote $[T_{k-1},T_k]$'s, a the partition of $[0,T]$, by the sub-intervals  on which arguments in the proof of Theorem \ref{3.2} are valid and satisfying the following two conditions: \\
(i) For each $k$,
\be  \label{nam1}
\norm{b}_{L^{q,1}([T_{k-1},T_k],L^p_x)} < \min\{K, C_0\}
\ee
(constant $C_0$ is from Remark \ref{stable} with $[T_{k-1},T_k]$ in place of $[0,T]$). \\
(ii) Solution $u^k$ constructed in \eqref{3.1} satisfies
\be  \label{nam2}
\norm{u^k}_{X^{q,p}([T_{k-1},T_k])}, \norm{\mathcal{M}(\nabla^2 u^k)}_{L^{q,1}([T_{k-1},T_k],L^p_x)} < \min\{\frac{K}{\sqrt{2N^2}}, \frac{K}{\sqrt{4C_1N^2}}\}
\ee
(constant $C_1$ is given by $C_1=16\bar{C}^4$, where $\bar{C}>1$ is a constant from Remark \ref{stable} with $[T_{k-1},T_k]$ in place of $[0,T]$).

Let us briefly explain what these conditions mean. First condition means that the stability estimate \eqref{stability estimate 1} of PDE \eqref{3.1} holds on each interval $[T_{k-1},T_k]$. Second condition says that $u^k$'s are small enough in some sense, which is a crucial assumption in order to apply the results in Remark \ref{re2.16}. It is possible to construct such partition by taking each sub-interval $[T_{k-1},T_k]$ sufficiently small.

Now, assume that not only $b_n$ converges to $b$ in $L^{q,1}([0,T], L^p_x)$, but also converges in the following sense: for each $k$,
\be \label{strong converge}
b_n \rightarrow b \qquad \text{in} \quad L^{q,1}([T_{k-1},T_k], L^p_x).
\ee 
For smooth drift $b_n$ satisfying \eqref{strong converge}, let $u^k_n$ be a solution to PDE \eqref{3.1}  with $b_n$ in place of  $b$. From \eqref{nam2} and \eqref{strong converge}, one can check that for each $k$,
\begin{multline} \label{123}
\limsup_n \norm{u^k_n}_{X^{q,p}([T_{k-1},T_k])},  \limsup_n \norm{\mathcal{M}(\nabla^2u^k_n)}_{L^{q,1}([T_{k-1},T_k],L^p_x)} \\
 < \min\{\frac{K}{\sqrt{2N^2}}, \frac{K}{\sqrt{4C_1N^2}}\}
\end{multline}
(see the condition \eqref{nam1} and  Remark \ref{stable}),
and $\Phi^k_n(t,x)=x+u^k_n(t,x)$ satisfy
\begin{gather*}
\frac{1}{2}< \norm{\nabla \Phi_n^k(t,x)}_{L^\infty([T_{k-1},T_k]\times \R^d)}<2, \quad \frac{1}{2}< \norm{\nabla^{-1} \Phi^k_n(t,x)}_{L^\infty([T_{k-1},T_k]\times \R^d)}<2.
\end{gather*}
 Let $\phi_n$ be a stochastic flow associated with the drift $b_n$ constructed as in the proof of Theorem \ref{3.2}. More precisely, $\phi_n$ is constructed on each interval $[T_{k-1},T_k]$, and then glued together. Under the condition \eqref{strong converge}, we show that the stochastic flow $\phi_n$ converges to $\phi$ in the following sense:
\begin{theorem} \label{3.5}
Suppose that smooth drifts $b_n$ converge to $b$ in the sense of \eqref{strong converge}, and the following quantity is uniformly bounded in $n$:
\be \label{dkdk}
\sup_{t\in [0,T],x\in \R^d}\E \exp\Big[2\int_0^tb_n(s,B^x_s)dB^x_s-\int_0^t b_n^2(s,B^x_s)ds\Big].
\ee Then, for any $r\in [1,\infty)$ and $x\in \R^d$, we have 
\be \label{stability}
\lim_{n\rightarrow \infty} \sup_{0\leq t\leq T}\E |\phi_n(0,t,x)-\phi(0,t,x)|^r=0.
\ee 
\end{theorem} 
In order to prove this theorem, we first show the statement of type \eqref{stability} for the conjugated SDE \eqref{conjugated}. We follow the arguments in \cite[Lemma 3]{ff2}, but due to the critical nature of exponents $p$ and $q$, the careful analysis is needed. We first prove this statement  for $r=1$, and later we will extend this to the general $r\in [1,\infty)$.
\bepro \label{3.6}
Let $Z^n$ and $Z$ be random variables and assume that smooth drifts $b_n$ converge to $b$ in the sense of \eqref{strong converge}. On each interval $ [T_{k-1},T_k]$, let us denote $X^n_t$ by a strong solution to SDE \eqref{SDE} with a drift $b_n$ and the initial condition $X_{T_{k-1}}^n=Z^n$, and similarly $X_t$ by a strong solution to SDE \eqref{SDE} with a drift $b$ and the initial condition $X_{T_{k-1}}=Z$. Then, for some constant $C$ independent of $Z^n$ and $Z$, 
\begin{align*}
\limsup_{n\rightarrow \infty} \sup_{T_{k-1}\leq t\leq T_k}\E |\Phi^k_n(t,X^n_t)-\Phi^k(t,X_t)| \leq C\limsup_{n\rightarrow \infty}\E|\Phi^k_n(T_{k-1},Z^n)-\Phi^k(T_{k-1},Z)|,
\end{align*} 
\be \label{stability 1} 
\limsup_{n\rightarrow \infty} \sup_{T_{k-1}\leq t\leq T_k}\E |X_t^n-X_t| \leq C\limsup_{n\rightarrow \infty} \E|Z_n-Z|.
\ee
\enpro 
\begin{proof}
Step 1. Proof the first inequality : without loss of the generality, let us only consider the case $T_{k-1}=0, T_k=T_1$. Throughout the proof, we use the simplified notations $u_n:=u^1_n$, $u:=u^1$, $\Phi_n:=\Phi^1_n$, $\Phi:=\Phi^1$, and $L^{q,1}_t(L^p_x):=L^{q,1}([0,T_1],L^p_x)$ (recall that $u^k_n$ is a solution to PDE \eqref{3.1} with $b_n$ in place of $b$). If we define that for $0\leq t\leq T_1$,
\begin{align*}
Y^n_t=\Phi_n(t,X^n_t),\quad Y_t=\Phi(t,X_t),
\end{align*} then $Y^n_t$, $Y_t$ are solutions to the conjugated SDE \eqref{conjugated} with 
\begin{align*}
\tilde{\sigma}_n(t,x)=I+\nabla u_n(t,\Phi_n^{-1}(t,x)),\quad  \tilde{\sigma}(t,x)=I+\nabla u(t,\Phi^{-1}(t,x)),
\end{align*} and the initial conditions $Y_0^n=\Phi_n(0,Z^n)$, $Y_0=\Phi(0,Z)$, respectively. Using It\^o's formula,
\begin{align*}
d|Y^n_t-Y_t|^2=\text{Trace}[(\nabla u_n(t,X^n_t)-\nabla u(t,X_t))(\nabla u_n(t,X^n_t)-\nabla u(t,X_t))^T]dt + dM_t
\end{align*}
for some martingale $M_t$ with a zero mean. The martingale property of $M_t$ can be easily verified using the boundedness of $\nabla u_n$ and $\nabla u$.
Note that due to Remark \ref{stable}, 
\begin{align*}
|\nabla u_n(t,X_t^n)-\nabla u(t,X_t&)|=|(\nabla u_n(t,X_t^n)-\nabla u_n(t,X_t))+(\nabla u_n(t,X_t)-\nabla u(t,X_t))| \\
&\leq \bar{C}( |\nabla u_n(t,X_t^n)-\nabla u_n(t,X_t)|+\norm{b_n-b}_{L^{q,1}_t(L^p_x)}).
\end{align*} 
Thus, we have
\begin{align}
&d|Y_t^n-Y_t|^2 \leq 2\bar{C}^2(|\nabla u_n(t,X_t^n)-\nabla u_n(t,X_t)|^2+\norm{b_n-b}_{L^{q,1}_t(L^p_x)}^2)dt +dM_t \nonumber\\
&= 2\bar{C}^2|X_t^n-X_t|^2 \frac{|\nabla u_n(t,X_t^n)-\nabla u_n(t,X_t)|^2}{|X_t^n-X_t|^2}dt +2\bar{C}^2\norm{b_n-b}_{L^{q,1}_t(L^p_x)}^2dt +dM_t \nonumber\\
&\leq 16\bar{C}^4|Y_t^n-Y_t|^2 dA^n_t+16\bar{C}^4\norm{b_n-b}_{L^{q,1}_t(L^p_x)}^2dA^n_t+16\bar{C}^4\norm{b_n-b}_{L^{q,1}_t(L^p_x)}^2dt+dM_t, \label{zzz}
\end{align}
 where an auxiliary process $A^n_t$ is defined by
\begin{align*}
dA^n_t=\mathds{1}_{X^n_t \neq X_t} \frac{|\nabla u_n(t,X_t^n)-\nabla u_n(t,X_t)|^2}{|X_t^n-X_t|^2}dt.
\end{align*} 
Note that in order to derive the inequality \eqref{zzz}, we used  the fact that
\begin{align*}
|Y_t^n-Y_t| &= |X_t^n+u_n(t,X_t^n)-X_t-u(t,X_t)| \\
&\geq |X^n_t+u(t,X^n_t)-X_t-u(t,X_t)|-|u_n(t,X^n_t)-u(t,X^n_t)| \\
&\geq \frac{1}{2}|X^n_t-X_t|-\norm{u_n-u}_{L^\infty} \geq \frac{1}{2}|X^n_t-X_t|-\bar{C} \norm{b_n-b}_{L^{q,1}_t(L^p_x)}
\end{align*}
(see  Remark \ref{stable} and the conditions \eqref{infinity}, \eqref{nam1}, \eqref{strong converge}), which implies that
\be \label{4}
 |X^n_t-X_t| \leq 2\bar{C}(|Y_t^n-Y_t|+\norm{b_n-b}_{L^{q,1}_t(L^p_x)}).
\ee 
Therefore, setting $C_1=16\bar{C}^4$, from \eqref{zzz},
\begin{align*}
d(e^{-C_1A^n_t}|Y_t^n-Y_t|^2)&=e^{-C_1A^n_t}d(|Y_t^n-Y_t|^2)-C_1e^{-C_1A^n_t}|Y_t^n-Y_t|^2dA^n_t \\
&\leq e^{-C_1A^n_t}[C_1\norm{b_n-b}_{L^{q,1}_t(L^p_x)}^2dA^n_t+C_1\norm{b_n-b}_{L^{q,1}_t(L^p_x)}^2dt+dM_t].
\end{align*}
Integrating in $t$ and then taking the expectation, we obtain
\begin{multline} \label{zzzz}
\E e^{-C_1A^n_t}|Y_t^n-Y_t|^2  \leq \E|\Phi_n(0,Z^n)-\Phi(0,Z)|^2 \\
+C_1\norm{b_n-b}_{L^{q,1}_t(L^p_x)}^2\E\Big[ \int_0^t e^{-C_1A^n_s} dA^n_s+\int_0^t e^{-C_1A^n_s} ds\Big].
\end{multline}
We now prove that 
\be \label{zzzzz}
\limsup_{n } E\Big[\int_0^{T_1} e^{-C_1A^n_t} dA^n_t\Big]<\infty.
\ee
Applying Proposition \ref{A.5}, we obtain
\begin{align*}
\E \int_0^{T_1} &e^{-C_1A^n_T}dA^n_t = \E \int_0^{T_1} e^{-C_1A^n_T}\frac{|\nabla u_n(t,X^n_t)-\nabla u_n(t,X_t)|^2}{|X_t^n-X_t|^2}dt \\
&\leq \E \int_0^{T_1} \frac{|\nabla u_n(t,X_t^n)-\nabla u_n(t,X_t)|^2}{|X^n_t-X_t|^2}dt \\
&\leq 2N^2\E \int_0^{T_1} (|\mathcal{M}(\nabla^2u_n)(t,X^n_t)|^2+|\mathcal{M}(\nabla^2u_n)(t,X_t)|^2) dt.
\end{align*}
Due to Remark \ref{re2.16}, for all sufficiently large $n$, the following quantities
\begin{align*}
\E\int_0^{T_1} 2N^2|\mathcal{M}(\nabla^2u_n)(t,X^n_t)|^2 dt,\quad \E\int_0^{T_1} 2N^2|\mathcal{M}(\nabla^2u_n)(t,X_t)|^2 dt
\end{align*} 
 are uniformly bounded since 
\begin{align*}
\sup_n \norm{\mathcal{M}(\nabla^2u_n)}_{L^{q,1}_t(L^p_x)}<\frac{K}{\sqrt{2N^2}}, \quad 
\limsup_n \norm{b_n}_{L^{q,1}_t(L^p_x)}<K, \quad \norm{b}_{L^{q,1}_t(L^p_x)}<K
\end{align*}
(see conditions \eqref{nam1}, \eqref{strong converge}, \eqref{123}, and \eqref{end2} in  Remark \ref{re2.16}).
Thus, we obtain \eqref{zzzzz}.

 Also, it is obvious that
\be \label{2}
\limsup_{n} \E \Big[\int_0^{T_1} e^{-C_1A^n_t} dt\Big] \leq T_1.
\ee
Furthermore, from the definition of $A^n_t$, we have
\begin{align*}
A^n_{T_1} \leq 2N^2\int_0^{T_1} (|\mathcal{M}(\nabla^2u_n)(t,X^n_t)|^2+|\mathcal{M}(\nabla^2u_n)(t,X_t)|^2) dt
\end{align*}
due to Proposition \ref{A.5}. Thanks to conditions \eqref{nam1}, \eqref{strong converge}, \eqref{123}, and Remark \ref{re2.16}, we have
\be \label{3}
\limsup_{n} \E e^{C_1A^n_{T_1}}<\infty.
\ee 
Therefore, applying \eqref{zzzz}, \eqref{zzzzz}, \eqref{2}, and \eqref{3} to the inequality
\begin{align*}
\E & |Y^n_t-Y_t| \leq [\E e^{-C_1A^n_t}|Y^n_t-Y_t|^2]^{1/2}[\E  e^{C_1A^n_t}]^{1/2},
\end{align*}
one can conclude the proof of the first statement of the proposition.

Step 2. Proof of \eqref{stability 1}: using \eqref{4}, on $t\in [T_{k-1},T_k]$,
\begin{align*}
|X^n_t-X_t| \leq 2\bar{C}(|\Phi^k_n(t,X^n_t)-\Phi^k(t,X_t)|+\norm{b_n-b}_{L^{q,1}([T_{k-1},T_k],L^p_x)}).
\end{align*}
 Combining this with the first statement of the proposition, for some constant $C$,
\begin{align*}
&\limsup_{n\rightarrow \infty} \sup_{T_{k-1}\leq t\leq T_k}\E |X_t^n-X_t| \leq C \limsup_{n\rightarrow \infty} \E |\Phi^k_n(T_{k-1},Z^n)-\Phi^k(T_{k-1},Z)| \\
&\leq C(\limsup_{n\rightarrow \infty} \E |\Phi^k_n(T_{k-1},Z^n)-\Phi^k_n(T_{k-1},Z)|+\limsup_{n\rightarrow \infty} \E |\Phi^k_n(T_{k-1},Z)-\Phi^k(T_{k-1},Z)|) \\
&\leq 2C \limsup_{n\rightarrow \infty} \E |Z^n-Z|.
\end{align*}  
 Here, we used the uniform Lipschitz continuity of $\Phi_n^k(t,\cdot)$ and the fact
\begin{align*}
\limsup_{n\rightarrow \infty} \E |\Phi^k_n(T_{k-1},Z)-\Phi^k(T_{k-1},Z)| &\leq \limsup_{n\rightarrow \infty} \norm{\Phi^k_n-\Phi^k}_{L^\infty} \\
&\leq C\limsup_{n\rightarrow \infty} \norm{b_n-b}_{L^{q,1}([T_{k-1},T_k],L^p_x)}=0
\end{align*}
which follows from the estimate \eqref{stability estimate 1}.

\end{proof} 

\begin{proof}[Proof of Theorem \ref{3.5}] 
When $r=1$, \eqref{stability} immediately follows from the estimate \eqref{stability 1} and the semigroup property of  the stochastic flow. For example, on the interval $[T_1,T_2]$,
\begin{align*}
\limsup_{n\rightarrow \infty} &\sup_{T_1\leq t\leq T_2}\E |\phi_n(0,t,x)-\phi(0,t,x)| \\
&=\limsup_{n\rightarrow \infty} \sup_{T_1\leq t\leq T_2}\E |\phi_n(T_1,t,\phi_n(0,T_1,x))-\phi(T_1,t,\phi(0,T_1,x))| \\
&\leq C\limsup_{n\rightarrow \infty}\E|\phi_n(0,T_1,x)-\phi(0,T_1,x)|=0.
\end{align*}
Similar argument works on each interval $[T_{k-1}, T_k]$ as well.
For general $r\in [1,\infty)$,
\begin{align*}
\E |\phi_n(0,t,x)&-\phi(0,t,x)|^r \leq [\E |\phi_n(0,t,x)-\phi(0,t,x)|]^{1/2}[\E |\phi_n(0,t,x)-\phi(0,t,x)|^{2r-1}]^{1/2}
\end{align*}
thanks to the H\"older's inequality. Note that
\begin{align*}
\E |\phi_n(0,t,x)-\phi(0,t,x)|^{2r-1}\leq C(\E |\phi_n(0,t,x)|^{2r-1}+\E |\phi(0,t,x)|^{2r-1}),
\end{align*}  and due to the Girsanov's theorem,
\begin{align*}
\E |\phi_n(0,t,x)|^{2r-1} &= \E \bigg[|B^x_t|^{2r-1}
\exp\Big[\int_0^tb_n(s,B^x_s)dB^x_s-\frac{1}{2}\int_0^tb_n^2(s,B^x_s)ds\Big]\bigg] \\
&\leq \E |B^x_t|^{4r-2}\cdot \E\exp\Big[2\int_0^tb_n(s,B^x_s)dB^x_s-\int_0^tb_n^2(s,B^x_s)ds\Big].
\end{align*} 
Thus, combining this with the uniform boundedness of the quantity \eqref{dkdk}, 
\begin{align*}
\sup_n \sup_{0\leq t\leq T}\E |\phi_n(0,t,x)-\phi(0,t,x)|^{2r-1} < \infty.
\end{align*}
Since we have already proved \eqref{stability} for $r=1$, the proof is completed.
\end{proof}
We now prove the main Theorem \ref{3.4}.  As in Proposition \ref{3.6}, we first show the Sobolev differentiablity of a solution $Y_t$ to the conjugated SDE \eqref{conjugated}. We introduce a refined notion of the convergence, which depends on the exponent $r$.
For given $1\leq r<\infty$, let us take sub-intervals $[T_{k-1}^r,T_k^r]$'s,  a partition of $[0,T]$, on which the arguments in the proof of Theorem \ref{3.2} are valid  and the following two conditions hold:
\be \label{123'}
 \norm{b}_{L^{q,1}([T_{k-1}^r,T_k^r],L^p_x)} < \min\{K,C_0\},
\ee 
\be \label{123''}
\norm{\nabla^2 u^k}_{L^{q,1}([T_{k-1}^r,T_k^r],L^p_x)}< \frac{K}{\sqrt{4r(2r-1)}}.
\ee
Here, $K=K(p,q,0,1)$ and $C_0$ are constants from  Remark \ref{re2.16} and Remark \ref{stable}, respectively. We say that smooth drifts $b_n$ \it{$r$-converge} to $b$ provided that for each $k$,
\be \label{strong r converge}
b_n \rightarrow b \qquad \text{in} \quad L^{q,1}([T_{k-1}^r,T_k^r], L^p_x).
\ee 
Note that due to the conditions \eqref{123'}, \eqref{123''}, and the stability result Remark \ref{stable}, we have
\be \label{last0}
\limsup_n \norm{\nabla^2 u^k_n}_{L^{q,1}([T_{k-1}^r,T_k^r],L^p_x)}< \frac{K}{\sqrt{4r(2r-1)}}.
\ee 
For $T_{k-1}^r\leq t\leq T_k^r$, let  us define $Y^{n,k}_t(y):=\Phi^k_n(t,X^n_t)$ to be a solution to the conjugated SDE \eqref{conjugated} starting from $y$ at $t=T_{k-1}^r$.
\bepro \label{3.7}
For each $r\in [1,\infty)$, suppose that smooth drifts $b_n$ $r$-converges to $b$ in the sense of \eqref{strong r converge}. Then, for each $k$, the quantity
\begin{align*}
\sup_{T_{k-1}^r\leq t\leq T_k^r}\sup_{y\in \R^d} \E |\nabla  Y^{n,k}_t(y)|^r 
\end{align*}
is uniformly bounded for all sufficiently large $n$.
\enpro
\begin{proof}
We follow the argument in \cite[Lemma 5]{ff2}. Without loss of the generality, let us consider the case $T_{k-1}^r=0$ and $T_k^r=T_1^r$, and use the simplified notations $u_n:=u_n^1$, $\Phi_n:=\Phi_n^1$,  $ Y^n:= Y^{n,k}$. Differentiating \eqref{conjugated}, we obtain
\begin{align*}
d(\nabla Y^n_t)=[\nabla^2u_n(t,\Phi_n^{-1}(t,Y^n_t))\nabla \Phi_n^{-1}(t,Y^n_t)\nabla Y^n_t]dB_t.
\end{align*}
Using the It\^o's formula, 
\be \label{ref r}
d|\nabla Y^n_t|^{2r} \leq 4r(2r-1)|\nabla Y^n_t|^{2r}|\nabla^2u_n(t,\Phi_n ^{-1}(t,Y^n_t)|^2dt+Z^n_tdB_t
\ee 
for some process $Z_t^n$ satisfying
\be \label{last1}
|Z^n_t| \leq C |\nabla^2 u_n(t,\Phi_n ^{-1}(t,Y^n_t))||\nabla Y^n_t|^{2r}
\ee 
for some universal constant $C$. Here, we used the fact that $\norm{\nabla \Phi_n^{-1}}_{L^\infty_{t,x}}<2$ (see Proposition \ref{2.8}).
If we define an auxilary process $A_t$ via
\begin{align*}
dA^n_t=|\nabla^2u_n(t,\Phi_n ^{-1}(t,Y^n_t))|^2dt,
\end{align*}
then by \eqref{ref r}, we have
\be \label{last}
d(\exp[-4r(2r-1)A^n_t]|\nabla Y^n_t|^{2r}) \leq \exp[-4r(2r-1)A_t^n]Z^n_tdB_t.
\ee 
Let $\tau_l$ be a stopping time defined by
\begin{align*}
\tau_l=\inf \{0\leq t\leq T^r_1 \ | \ |\nabla Y^n_t| >l\},
\end{align*}
and $\tau_l=T^r_1$ if the above set is empty ($\tau_l$ depends on $n$, but we drop the index $n$ to alleviate the notation).
Integrating \eqref{last} in $t$ and then taking the expectation, we have
\be \label{last2}
\E\big[\exp[-4r(2r-1)A^n_{t \wedge \tau_l}]|\nabla Y^n_{t \wedge \tau_l}|^{2r}\big]\leq d^r +\E \int_0^t \exp[-4r(2r-1)A_s^n]Z^n_s \mathds{1}_{s\leq \tau_l} dB_s
\ee 
since $\nabla Y^n_0 = I$ (recall that $|\cdot|$ denotes a Hilbert-Schmidt norm). Note that according to Lemma \ref{2.11} and \eqref{last1}, for each $l$,
\begin{align*}
\int_0^t \E \big[\exp[-4r(2r-1)A_s^n]Z^n_s \mathds{1}_{s\leq \tau_l}\big]^2ds \leq C^2 l^{4r}\int_0^t \E |\nabla^2 u_n(t,X^n_t)|^2 ds<\infty.
\end{align*}
This implies that the second term of RHS in \eqref{last2} is equal to zero. Thus, thanks to Fatou's lemma and \eqref{last2},
\begin{align*}
\E \big[\exp[-4r(2r-1)A^n_t]|\nabla Y^n_t|^{2r}\big] \leq \liminf_{l\rightarrow \infty} \E \big[\exp[-4r(2r-1)A^n_{t \wedge \tau_l}]|\nabla Y^n_{t \wedge \tau_l}|^{2r}\big] \leq d^r.
\end{align*}
Using the H\"older's inequality,
\begin{align*}
\E|\nabla Y^n_t|^r &\leq \big[\E \exp[-4r(2r-1)A^n_t]|\nabla Y^n_t|^{2r}\big]^{\frac{1}{2}}\big[\E \exp[4r(2r-1)A^n_t]\big]^{\frac{1}{2}} \\
&\leq d^{\frac{r}{2}} \E\big[ \exp[4r(2r-1)A^n_t]\big]^{\frac{1}{2}}.
\end{align*}
Due to the conditions \eqref{123'} and \eqref{last0}, for all sufficiently large $n$, the quantity
\begin{align*}
\E \exp[4r(2r-1)A^n_{T_1^r}]= \E \exp\Big[4r(2r-1)\int_0^{{T_1^r}}|\nabla^2u_n(s,X^n_s)|^2ds \Big]
\end{align*} is uniformly bounded (see \eqref{end2} in  Remark \ref{re2.16}). This concludes the proof.
\end{proof}
\begin{proof}[Proof of Theorem \ref{3.4}]
Fix $r\in [1,\infty)$ and then choose a partition $[T_{k-1}^r,T_k^r]$ of $[0,T]$ satisfying \eqref{nam1}, \eqref{nam2}, \eqref{123'}, and \eqref{123''}. Let us choose a smooth approximation $b_n$ to $b$ satisfying the following two conditions; \\
(i) $b_n$ converges to $b$ in $L^{q,1}([T_{k-1}^r,T_k^r], L^p_x)$ for each $k$, \\
(ii) the following quantity is uniformly bounded in $n$:
\be \label{1234}
\sup_{t\in [0,T],x\in \R^d}\E \exp\Big[2\int_0^tb_n(s,B^x_s)dB^x_s-\int_0^t b_n^2(s,B^x_s)ds\Big].
\ee 
It is possible to choose such approximation once we recall the proof of Lemma \ref{2.10} and Proposition \ref{2.3}. For all sufficiently large $n$,
\begin{align*}
\sup_{0\leq t\leq T}\sup_{x\in \R^d} \E |\nabla \phi_n(0,t,x)|^r
\end{align*}
is uniformly bounded due to Proposition \ref{3.7}, semigroup property, and the uniform boundedness of $\nabla \Phi_n$, $\nabla \Phi_n^{-1}$. Thus, for any $t\in [0,T]$,  there exist a random field $\Psi$ such that
\begin{align*}
\nabla \phi_n(0,t,\cdot) \rightharpoonup \Psi \quad \text{weak-*} \ \text{in} \ L^\infty( \R^d,L^r(\Omega))
\end{align*} 
up to an appropriate subsequence.
 From this, we will show that $\phi(0,t,\cdot)$ is almost surely  weakly differentiable, and its weak derivative is $\Psi$. For any test function $\varphi\in C^\infty_c(\R^d)$ and random variable $Z\in L^\infty(\Omega)$,
\begin{align}
\E \Big[ \big(\int_{\R^d} \Psi \varphi(x)dx\big)Z\Big] &= \lim_{n\rightarrow \infty} \E \Big[\big(\int_{\R^d} \nabla \phi_n(0,t,x)\varphi(x)dx\big)Z\Big] \nonumber \\
&= -\lim_{n\rightarrow \infty} \E \Big[\big(\int_{\R^d} \phi_n(0,t,x\big)\nabla \varphi(x)dx)Z\Big]  \nonumber \\
&= -\E \Big[\big(\int_{\R^d} \phi(0,t,x)\nabla \varphi(x)dx\big)Z\Big]. \label{z}
\end{align}
Let us check the validity of the last line \eqref{z} of the above identities. Note that Theorem \ref{3.5} implies that for each $x\in \R^d$,
\be \label{ldct} 
\E [\phi_n(0,t,x) \nabla \varphi(x) Z] \rightarrow \E [\phi(0,t,x) \nabla \varphi(x) Z]
\ee 
as $n \rightarrow \infty$. Also, according to the Girsanov theorem and H\"older's inequality, we have
\begin{align} \label{440}
|\E&[\phi_n(0,t,x)Z\nabla\varphi(x)]|\leq C \E|\phi_n(0,t,x)| \nonumber \\
&= C\E \bigg[|x+B_t|\cdot \exp\Big[\int_0^tb_n(s,B^x_s)dB^x_s-\frac{1}{2}\int_0^tb_n^2(s,B^x_s)ds\Big]\bigg] \nonumber \\
&\leq C\Big[\E |x+B_t|^{2}\Big]^{1/2}\cdot \bigg[\E \exp\Big[ 2\int_0^tb_n(s,B^x_s)dB^x_s-\int_0^tb_n^2(s,B^x_s)ds\Big]\bigg]^{1/2}.
\end{align}
It is obvious that for any compact set $K$ in $\R^d$,
\begin{align*}
\sup_{x\in K} \E |x+B_t|^{2} <\infty.
\end{align*} 
Since the quantity \eqref{1234} is uniformly bounded in $n$ and $\varphi$ has compact support, from \eqref{440}, we obtain
\begin{align*}
 \sup_n |\E[\phi_n(0,t,x)Z\nabla \varphi(x)]| \in L^1(\R^d).
\end{align*} 
Thus, \eqref{z}  follows from \eqref{ldct} and the Lebesgue dominated convergence theorem.

Therefore, from \eqref{z}, since $Z\in L^\infty(\Omega)$ is arbitrary, $w$-almost surely,
\begin{align*}
\int_{\R^d} \Psi\varphi(x)dx=-\int_{\R^d} \phi(0,t,x)\nabla \varphi(x)dx
\end{align*} 
holds for any $\varphi\in C^\infty_0(\R^d)$.
This immediately implies that the weak derivative of $\phi(0,t,\cdot)$ is equal to $\Psi$. Since $\Psi \in L^\infty(\R^d, L^r(\Omega))$,  we have
\begin{align*}
\sup_{x\in \R^d} \E |\nabla \phi(0,t,\cdot)|^r<\infty.
\end{align*}
This concludes the proof.
\end{proof}

\appendix
\section{Lorentz spaces and some lemmas} \label{section a}
In this appendix, we recall some useful properties about the  Lorentz spaces. Also, we introduce some useful lemmas  used frequently in this paper.
\begin{definition}
 (Lorentz spaces). A complex-valued function $f$ defined on the measure space $(X,\mu)$ belongs to the \it{Lorentz space} $L^{p,q}(X,d\mu)$ if the quantity
\be 
\norm{f}_{L^{p,q}(X)}:= p^{\frac{1}{q}}\norm{t\mu(|f|\geq t)^{\frac{1}{p}}}_{L^q(\R^+, \frac{dt}{t})}
\ee 
is finite.
\end{definition}
The concept of Lorentz spaces is introduced in \cite{lorentz}. These spaces can be regarded as  generalizations of the standard Lebesgue $L^p(X,d\mu)$ spaces. In the case when $q=p$, $L^{p,p}$ coincides with the standard $L^p$ spaces, and when $q=\infty$, $L^{p,\infty}$ coincides with the weak $L^p$ spaces. Lorentz spaces are quasi-Banach spaces in the sense that for some constant $c=c(p,q)>1$,
\be \label{quasi}
\norm{f+g}_{L^{p,q}} \leq c(\norm{f}_{L^{p,q}}+\norm{g}_{L^{p,q}})
\ee
for any $f,g\in L^{p,q}$, and it is complete with respect to $\norm{\cdot}_{L^{p,q}}$. Also, Lorentz spaces can be realized as a real interpolation of two $L^p$ spaces: for the exponents $1<p,p_1,p_2<\infty$, $0<\theta<1$, $1\leq q\leq \infty$ satisfying $\frac{1}{p}=\frac{1-\theta}{p_1}+\frac{\theta}{p_2}$,
\begin{align*}
[L^{p_1},L^{p_2}]_{\theta,q}=L^{p,q},
\end{align*} 
where $[\cdot,\cdot ]_{\theta,q}$ denotes the real interpolation (see \cite{interpolation} for details). 
\begin{remark} \label{appre}
From the definition of Lorentz spaces, we can easily check that the following property holds: if $p<\infty$, then for any $\epsilon>0$, there exists $\delta>0$ such that
\begin{align*}
\norm{f}_{L^{p,q}(A)} < \epsilon
\end{align*}
for all measurable set $A\subseteq X$ satisfying $\mu(A)<\delta$. 
Also, one can check that for any two disjoint measurable sets $A,B \subseteq X$ and $f\in L^{p,q}(X)$,
\begin{align*}
\norm{f}_{L^{p,q}(A)}+\norm{f}_{L^{p,q}(B)} \sim_{p,q} \norm{f}_{L^{p,q}(A \cup B)}.
\end{align*} 
\end{remark}
The following lemma is used to prove  Proposition \ref{2.7}.
\begin{lemma} \label{A.2}
Let us denote $P(t,x)$ by the standard heat kernel. Then, $\nabla P\in L^{q,\infty}(\R, L^p_x)$ for any exponents $p,q \in (1,\infty)$ satisfying $\frac{2}{q}+\frac{d}{p}=d+1$.
\end{lemma}
\begin{proof} Note that
\begin{align*}
\Big \vert D_{x_j}(\frac{1}{t^{d/2}}e^{-|x|^2/4t})\Big \vert = \Big \vert \frac{x_j}{2t}\frac{1}{t^{d/2}}e^{-|x|^2/4t}\Big \vert \leq \frac{|x|}{2t^{(d+2)/2}}e^{-|x|^2/4t}.
\end{align*}
Therefore,  using the condition $\frac{2}{q}+\frac{d}{p}=d+1$, for some constant $C=C(p,q)$,
\begin{align*}
\norm{\nabla P}_{L^{q,\infty}_t(L^p_x)} &\leq \norm{\norm{\frac{|x|}{2t^{(d+2)/2}}e^{-|x|^2/4t}}_{L^p_x}}_{L^{q,\infty}_t} =C\norm{\frac{t^{(p+d)/2p}}{2t^{(d+2)/2}}}_{L^{q,\infty}_t} = C\norm{\frac{1}{2}t^{-1/q}}_{L^{q,\infty}_t}<\infty.
\end{align*}

\end{proof}
 There are counterparts of the H\"older's and Young's inequalities for the Lorentz spaces. H\"older's inequality for the Lorentz spaces claims that for $1\leq p_1,p_2,p<\infty$ , $0<q_1,q_2,q\leq \infty$ satisfying $\frac{1}{p}=\frac{1}{p_1}+\frac{1}{p_2}$ and $\frac{1}{q}=\frac{1}{q_1}+\frac{1}{q_2}$,
\begin{align*}
\norm{fg}_{L^{p,q}(X,d\mu)} \leq C(p,q,p_1,q_1,p_2,q_2) \norm{f}_{L^{p_1,q_1}(X,d\mu)}\norm{g}_{L^{p_2,q_2}(X,d\mu)}.
\end{align*}
O'Neil's convolution inequality \cite{convolution} claims that for $1<p_1,p_2<\infty$, $0<q_1,q_2<\infty$ satisfying $1+\frac{1}{p}=\frac{1}{p_1}+\frac{1}{p_2}$  and $\frac{1}{q}=\frac{1}{q_1}+\frac{1}{q_2}$,
\begin{align*}
\norm{f*g}_{L^{p,q}(\R^d,dx)} \leq C(p,q,p_1,q_1,p_2,q_2) \norm{f}_{L^{p_1,q_1}(\R^d,dx)}\norm{g}_{L^{p_2,q_2}(\R^d,dx)}.
\end{align*} 
One can extend the O'Neil's convolution inequality to the mixed-norm Lorentz spaces. We in particular consider the case $p=q=\infty$ for our purposes (see Proposition \ref{2.7}):
\bepro \label{A.3}
Suppose that $p_1,p_2,q_1,q_2\in (1,\infty)$ and $r_1,r_2,s_1,s_2\in [1,\infty]$ satisfy $\frac{1}{p_1}+\frac{1}{p_2}=\frac{1}{q_1}+\frac{1}{q_2}=1$ and $\frac{1}{r_1}+\frac{1}{r_2}= \frac{1}{s_1}+\frac{1}{s_2}= 1$. Then, for any $f\in L^{q_1,r_1}(\R,L^{p_1,s_1}(\R^d))$ and $g\in L^{q_2,r_2}(\R,L^{p_2,s_2}(\R^d))$, 
\begin{align*}
\norm{f*g}_{L^\infty_{t,x}} \leq C(p_1,p_2,q_1,q_2,r_1,r_2,s_1,s_2) \norm{f}_{L^{q_1,r_1}_t(L^{p_1,s_1}_x)}\norm{g}_{L^{q_2,r_2}_t(L^{p_2,s_2}_x)}.
\end{align*}
\enpro
\begin{proof}
 Note that 
\begin{align*}
|f*g|(t,x)\leq \int_{\R}\int_{\R^d} |f(s,y)g(t-s,x-y)|dyds=\norm{f(\cdot,\cdot)g(t-\cdot,x-\cdot)}_{L^1_t(L^1_x)}.
\end{align*}
 Since $\norm{g}_{L^{q_2,r_2}_t(L^{p_2,s_2}_x)}$ is invariant under the operations $g(\cdot) \mapsto g(c+\cdot)$ and $g(\cdot) \mapsto g(-\cdot)$, it suffices to prove that 
\begin{align*}
\norm{fg}_{L^1_{t,x}} \leq C \norm{f}_{L^{q_1,r_1}_t(L^{p_1,s_1}_x)}\norm{g}_{L^{q_2,r_2}_t(L^{p_2,s_2}_x)}.
\end{align*} 
Using  H\"older's inequality for the Lorentz spaces, we obtain
\begin{align*}
\norm{fg}_{L^1_{t,x}}&=\int_{\R}\int_{\R^d}|f|(t,x)|g|(t,x)dxdt \\
&\leq C \int_{\R}\norm{f(t,\cdot)}_{L^{p_1,s_1}_x}\norm{g(t,\cdot)}_{L^{p_2,s_2}_x}dt \leq C \norm{f}_{L^{q_1,r_1}_t(L^{p_1,s_1}_x)}\norm{g}_{L^{q_2,r_2}_t(L^{p_2,s_2}_x)}.
\end{align*}
\end{proof}
We need a slight extension of the standard Banach fixed point theorem to the quasi-Banach spaces, since the Lorentz spaces are quasi-Banach spaces.
\bepro \label{A.4}
Suppose that $X$ is a quasi-Banach space, and for some $c> 1$,
\begin{align*}
\norm{x+y} \leq c(\norm{x}+\norm{y})
\end{align*} 
hold for any $x,y\in X$. Also, assume that for some $\theta>0$ satisfying $c\theta
<1$, a map $T:X\rightarrow X$ satisfy that for any $x,y\in X$,
\begin{align*}
|T(x)-T(y)| \leq \theta |x-y|.
\end{align*} 
Then, $T$ has a unique fixed point.
\enpro
\begin{proof}
The proof of Proposition \ref{A.4} is similar to the standard proof of Banach fixed point theorem. Choose an arbitrary $x_0\in X$ and let us define $x_n:=T(x_{n-1})$ inductively for $n\geq 1$. It is obvious that
\begin{align*}
d(x_{n+1},x_n) \leq \theta ^nd(x_1,x_0).
\end{align*}
Using a quasi-norm property of $X$, for any $m>n$,
\begin{align*}
d(x_m,x_n)&\leq cd(x_m,x_{n+1})+cd(x_{n+1},x_n) \\ &\leq c^2d(x_m,x_{n+2})+c^2d(x_{n+2},x_{n+1})+cd(x_{n+1},x_n)\\ &\leq \cdots \\ &\leq c^{m-(n+1)}d(x_m,x_{m-1})+\sum_{k=1}^{m-(n+1)}c^kd(x_{n+k},x_{n+k-1})\\ &\leq \Big[c^{m-(n+1)}\theta^{m-1}+\sum_{k=1}^{m-(n+1)}c^k\theta^{n+k-1}\Big]d(x_1,x_0) \\
&<  ((c \theta )^{m-1} c^{-n} +(1-c\theta)^{-1} c^{-(n-1)})d(x_1,x_0).
\end{align*}
This implies that $\{x_n\}$ is a Cauchy sequence, thus it converges to a limit $x^*$ in $X$ since $(X,d)$ is complete. Since $T$ is continuous, we can readily check that $x^*$ is a fixed point. Uniqueness is obvious.
\end{proof}

Now, we introduce some useful lemmas  used in the paper. 
\bepro \label{A.5}
Let us denote $\mathcal{M}$ by the Hardy-Littlewood maximal function. Then, there exists a constant $N=N(d)$ such that the following property holds: for any $u\in C^\infty(\R^d)$ and $x,y\in \R^d$,
\begin{align*}
|u(x)-u(y)| \leq N|x-y|(\mathcal{M}|\nabla u|(x)+\mathcal{M}|\nabla u|(y)).
\end{align*}
\enpro

The last proposition is a useful criteria to derive a global bijectivity of the map, which is called the \it{Hadamard lemma} (see \cite[Theorem V.59]{hadamard}).
\bepro \label{A.7} 
Suppose that a $C^k$($k\geq 1$) map $F:\R^d \rightarrow \R^d$ satisfies the following properties: \\
(i) $\nabla F(x)$ is non-singular for every $x\in \R^d$, \\
(ii) $\lim_{|x|\rightarrow \infty}|F(x)|=\infty$. \\
Then, $F$ is a $C^k$ diffeomorphism from $\R^d$ to itself. 
\enpro

\section*{Acknowledgement}
The author thanks to the advisor Fraydoun Rezakhanlou for introducing this problem and sharing interesting ideas.

\end{document}